\documentclass[reqno]{amsart}

\usepackage[all]{xy}
\usepackage{graphicx}

\usepackage{amssymb}
\usepackage{amsmath}
\usepackage{mathrsfs}
\usepackage{epsfig}
\usepackage{amscd}
\usepackage{graphicx, color}

\def\E{\ifmmode{\mathbb E}\else{$\mathbb E$}\fi} 
\def\N{\ifmmode{\mathbb N}\else{$\mathbb N$}\fi} 
\def\R{\ifmmode{\mathbb R}\else{$\mathbb R$}\fi} 
\def\Q{\ifmmode{\mathbb Q}\else{$\mathbb Q$}\fi} 
\def\C{\ifmmode{\mathbb C}\else{$\mathbb C$}\fi} 
\def\Z{\ifmmode{\mathbb Z}\else{$\mathbb Z$}\fi} 
\def\P{\ifmmode{\mathbb P}\else{$\mathbb P$}\fi} 
\def\T{\ifmmode{\mathbb T}\else{$\mathbb T$}\fi} 
\def\SS{\ifmmode{\mathbb S}\else{$\mathbb S$}\fi} 

\def\DD{\ifmmode{\mathbb D}\else{$\mathbb D$}\fi} 

\newcommand{\del}{\partial}
\newcommand{\Cont}{{\operatorname{Cont}}}

\newcommand{\ben}{\begin{enumerate}}
\newcommand{\een}{\end{enumerate}}
\newcommand{\be}{\begin{equation}}
\newcommand{\ee}{\end{equation}}
\newcommand{\bea}{\begin{eqnarray}}
\newcommand{\eea}{\end{eqnarray}}
\newcommand{\beastar}{\begin{eqnarray*}}
\newcommand{\eeastar}{\end{eqnarray*}}
\newcommand{\bc}{\begin{center}}
\newcommand{\ec}{\end{center}}

\theoremstyle{theorem}
\newtheorem{thm}{Theorem}[section]
\newtheorem{cor}[thm]{Corollary}
\newtheorem{lem}[thm]{Lemma}
\newtheorem{prop}[thm]{Proposition}

\theoremstyle{definition}
\newtheorem{defn}[thm]{Definition}
\newtheorem{rem}[thm]{Remark}

\newtheorem{hypo}[thm]{Hypothesis}
\newtheorem{choice}[thm]{Choice}

\newtheorem{observation}[thm]{Observation}

\newtheorem*{thm*}{Theorem}

\numberwithin{equation}{section}

\def\R{{\mathbb R}}
\def\osc{{\hbox{\rm osc }}}
\def\Crit{{\hbox{Crit}}}

\def\E{{\mathbb E}}
\def\Z{{\mathbb Z}}
\def\C{{\mathbb C}}
\def\R{{\mathbb R}}
\def\P{{\mathbb P}}

\def\N{{\mathbb N}}

\def\11{{\mathbb I}}

\def\delbar{{\overline \partial}}

\def\Fix{{\text{\rm Fix}}}

\def\id{{\text{\rm id}}}

\def\H{\mathbb{H}}

\def\C{\mathbb{C}}
\def\Z{\mathbb{Z}}

\def\T{\mathbb{T}}

\def\Q{\mathbb{Q}}

\def\E{\ifmmode{\mathbb E}\else{$\mathbb E$}\fi} 
\def\N{\ifmmode{\mathbb N}\else{$\mathbb N$}\fi} 
\def\R{\ifmmode{\mathbb R}\else{$\mathbb R$}\fi} 
\def\Q{\ifmmode{\mathbb Q}\else{$\mathbb Q$}\fi} 
\def\C{\ifmmode{\mathbb C}\else{$\mathbb C$}\fi} 
\def\Z{\ifmmode{\mathbb Z}\else{$\mathbb Z$}\fi} 
\def\P{\ifmmode{\mathbb P}\else{$\mathbb P$}\fi} 
\def\SS{\ifmmode{\mathbb S}\else{$\mathbb S$}\fi} 
\def\DD{\ifmmode{\mathbb D}\else{$\mathbb D$}\fi} 

\def\R{{\mathbb R}}
\def\osc{{\hbox{\rm osc}}}
\def\Crit{{\hbox{Crit}}}
\def\E{{\mathbb E}}
\def\Z{{\mathbb Z}}
\def\C{{\mathbb C}}
\def\R{{\mathbb R}}

\def\N{{\mathbb N}}

\def\CL{{\mathcal L}}

\def\ev{{\text{\rm ev}}}

\def\delbar{{\overline \partial}}

\def\CA{{\mathcal A}}
\def\CB{{\mathcal B}}
\def\CC{{\mathcal C}}

\def\CL{{\mathcal L}}
\def\CM{{\mathcal M}}

\def\CP{{\mathcal P}}

\def\CP{{\mathcal P}}

\def\Dev{\operatorname{Dev}}

\def\Image{\operatorname{Image}}

\def\Ev{\text{\rm Ev}}

\def\mq{\mathfrak{q}}
\def\mp{\mathfrak{p}}
\def\mH{\mathfrak{H}}
\def\mh{\mathfrak{h}}
\def\ma{\mathfrak{a}}
\def\ms{\mathfrak{s}}
\def\mm{\mathfrak{m}}
\def\mn{\mathfrak{n}}
\def\mz{\mathfrak{z}}
\def\mw{\mathfrak{w}}
\def\Hoch{{\tt Hoch}}
\def\mt{\mathfrak{t}}
\def\ml{\mathfrak{l}}
\def\mT{\mathfrak{T}}
\def\mL{\mathfrak{L}}
\def\mg{\mathfrak{g}}
\def\md{\mathfrak{d}}
\def\mr{\mathfrak{r}}

\def\dim{\operatorname{dim}}

\def\Dev{\operatorname{Dev}}

\def\Cont{\operatorname{Cont}}
\def\Crit{\operatorname{Crit}}
\def\Spec{\operatorname{Spec}}
\def\Sing{\operatorname{Sing}}
\def\GFQI{\mathfrak{G}}
\def\Index{\operatorname{Index}}

\def\Image{\operatorname{Image}}
\def\ev{\operatorname{ev}}

\def\ben{\begin{enumerate}}
\def\een{\end{enumerate}}
\def\be{\begin{equation}}
\def\ee{\end{equation}}
\def\bea{\begin{eqnarray}}
\def\eea{\end{eqnarray}}
\def\beastar{\begin{eqnarray*}}
\def\eeastar{\end{eqnarray*}}
\def\bc{\begin{center}}
\def\ec{\end{center}}

\begin{document}

\quad \vskip1.375truein

\def\mq{\mathfrak{q}}
\def\mp{\mathfrak{p}}
\def\mH{\mathfrak{H}}
\def\mh{\mathfrak{h}}
\def\ma{\mathfrak{a}}
\def\ms{\mathfrak{s}}
\def\mm{\mathfrak{m}}
\def\mn{\mathfrak{n}}
\def\mz{\mathfrak{z}}
\def\mw{\mathfrak{w}}
\def\Hoch{{\tt Hoch}}
\def\mt{\mathfrak{t}}
\def\ml{\mathfrak{l}}
\def\mT{\mathfrak{T}}
\def\mL{\mathfrak{L}}
\def\mg{\mathfrak{g}}
\def\md{\mathfrak{d}}
\def\mr{\mathfrak{r}}
\def\Cont{\operatorname{Cont}}
\def\Crit{\operatorname{Crit}}
\def\Spec{\operatorname{Spec}}
\def\Sing{\operatorname{Sing}}
\def\GFQI{\text{\rm g.f.q.i.}}
\def\Index{\operatorname{Index}}
\def\Cross{\operatorname{Cross}}
\def\Ham{\operatorname{Ham}}
\def\Fix{\operatorname{Fix}}
\def\Graph{\operatorname{Graph}}
\def\id{\text{\rm id}}

\title[Proof of Shelukhin's conjecture]
{Contact instantons, anti-contact involution and proof of Shelukhin's conjecture}

\author{Yong-Geun Oh}
\address{Center for Geometry and Physics, Institute for Basic Science (IBS),
77 Cheongam-ro, Nam-gu, Pohang-si, Gyeongsangbuk-do, Korea 790-784
\& POSTECH, Gyeongsangbuk-do, Korea}
\email{yongoh1@postech.ac.kr}
\thanks{This work is supported by the IBS project \# IBS-R003-D1}

\begin{abstract}  In this paper, we prove Shelukhin's conjecture
on the translated points on \emph{any closed} contact manifold $(Q,\xi)$ which reads that
for any choice of function $H = H(t,x)$ and contact form $\lambda$ the contactomorphism $\psi_H^1$
carries a translated point in the sense of Sandon, whenever the inequality
$$
\|H\| \leq T(\lambda,M)
$$
holds the case. Main geometro-analytical tools are those of bordered contact instantons
employed in \cite{oh:entanglement1} with  Legendrian boundary condition via the Legendrianization
of contact diffeomorphisms. Along the way, we 
utilize the functorial construction of the \emph{contact product} that carries an involutive
symmetry and develop relevant contact Hamiltonian geometry with involutive symmetry. This involutive
symmetry plays a fundamental role in our proof in combination with the analysis of contact instantons.
\end{abstract}


\maketitle

\tableofcontents

\section{Introduction}

Let $(M,\xi)$ be a coorientable contact manifold, which will be assumed throughout
the paper without further mentioning, unless otherwise said.

The notion of \emph{translated points} was introduced by
Sandon \cite{sandon:translated} in the course of her applications of
Legendrian spectral invariants to the problem of contact nonsqueezing
initiated by Eliashberg-Kim-Polterovich \cite{eliash-kim-pol} and further studied by
\cite{chiu}, \cite{fraser}.

We first recall the definition of translated points.
For each given coorientation preserving contact diffeomorphism $\psi$ of $(M,\xi)$, we call
the function $g$ appearing in $\psi^*\lambda = e^g \lambda$ the \emph{conformal exponent}
for $\psi$ and denote it by $g=g_\psi$ \cite{oh:contacton-Legendrian-bdy}.

\begin{defn}[Sandon \cite{sandon:translated}] Let $(M,\xi)$ be a
contact manifold equipped with a contact form $\lambda$.
A point $x \in M$ is called a $\lambda$-translated point of a contactomorphism $\psi$
if $x$ satisfies
$$
\begin{cases}
g_\psi(x) = 0 \\
\psi(x) = \phi_{R_\lambda}^\eta(x)\,  \text{\rm for some $\eta \in \R$}.
\end{cases}
$$
We denote the set of $\lambda$-translated points of $\psi$ by $\Fix_\lambda^{\text{\rm trn}}(\psi)$.
\end{defn}

The present paper is a continuation of the present author's paper
\cite{oh:entanglement1}  which introduces the general problem of 
\emph{entanglement} of Legendrian links, which 
also provides the geometro-analytical framework for the present paper.
However
the nature of the present paper is largely geometric in the spirit of contact dynamics and
contact Hamiltonian geometry. The heart of the matter in our
 proofs of the present paper lies in the quantitative study of
the moduli space of perturbed contact instantons and its interplay 
with the contact Hamiltonian geometry and its $\Z_2$-involution symmetry of
the contact instanton homology introduced in \cite{oh:entanglement1,oh:contacton-gluing,
oh:entanglement2}.

Other analytical foundations
such as a priori elliptic estimates, Gromov-Floer-Hofer style convergence, 
relevant Fredholm theory and the generic mapping and evaluation transversalities
are already established in \cite{oh:entanglement1}, \cite{oh:entanglement1},
\cite{oh:contacton-gluing}, \cite{oh:perturbed-contacton} and
\cite{oh:contacton-transversality} which
 will not be repeated in the present paper but
referred thereto. Only the essential components are proved and
others are summarized in Appendix for readers' convenience and for the self-containedness of
the paper. Frequently we will be brief
leaving only the core of the argument
when we mention these analytical details. We also
use the same notations and terminologies therefrom.
(We refer readers to the survey article 
\cite{oh-kim:survey} for some coherent exposition of these
analytic foundations of (perturbed) contact instantons,
 and with some comparison with the framework of
 pseudoholomorphic curves on the symplectization.)
  
\subsection{Statement of main result}

We consider the following union
\be\label{eq:S-trace}
Z_S: = \bigcup_{t \in \R} \phi_{R_\lambda}^t(S)
\ee
which is called the \emph{Reeb  trace} of a subset $S\subset M$ in \cite{oh:entanglement1}.

We recall the following standard definition $T(M,R;\lambda)$ in contact geometry.

\begin{defn}\label{defn:spectrum} Let $\lambda$ be a contact form of contact manifold $(M,\xi)$
and $R \subset M$ a connected Legendrian submanifold.
Denote by $\mathfrak{Reeb}(M,\lambda)$ (resp. $\mathfrak{Reeb}(M,R;\lambda)$) the set of closed Reeb orbits
(resp. the set of self Reeb chords of $R$).
\begin{enumerate}
\item
We define $\operatorname{Spec}(M,\lambda)$ to be the set
$$
\operatorname{Spec}(M,\lambda) = \left\{\int_\gamma \lambda \mid \gamma \in \mathfrak Reeb(M,\lambda)\right\}
$$
and call the \emph{action spectrum} of $(M,\lambda)$.
\item We define the \emph{period gap} to be the constant given by
$$
T(M,\lambda): = \inf\left\{\int_\gamma \lambda \mid \lambda \in \mathfrak Reeb(M,\lambda)\right\} > 0.
$$
\end{enumerate}
We define
$$
\operatorname{Spec}(M,R;\lambda): =\left \{ \int_\gamma \lambda \mid \gamma \in \mathfrak{Reeb}(M,R;\lambda)\right\}
$$
and the associated $T(M,R;\lambda)$ similarly. We define
\be\label{eq:TMR}
T_\lambda(M;R): = \min\{T(M,\lambda), T(M,R;\lambda)\}
\ee
and call it the \emph{(chord) period gap} of $R$ in $M$.
\end{defn}
We denote by $\Cont(M,\xi)$ the set of contactomorphisms and by $\Cont_0(M,\xi)$ its
identity component. We also use the notation $H \mapsto \psi$ when $\psi = \psi_H^1$
for a given $\psi \in \Cont_0(M,\xi)$.  We recall the definition of 
the $L^{(1,\infty)}$-norm of $H$ defined by
\be\label{eq:||H||-intro}
\|H\|: =  \int_0^1  \max H_t - \min H_t \, dt.
\ee
The main goal of the present paper is to give the proof of
the following conjecture of Shelukhin  \cite[Conjecture 31]{shelukhin:contactomorphism}.

\begin{thm}[Main Theorem; Shelukhin's conjecture]\label{thm:translated-intro}
Assume $(M,\xi)$ is a compact contact manifold and let $\psi \in \Cont_0(M,\xi)$.
Let $\lambda$ be a contact form of $\xi$ and denote by $T(M,\lambda) > 0$ the minimal
period of closed Reeb orbits of $\lambda$. Then for any $H \mapsto \psi$, the following holds:
\begin{enumerate}
\item Provided $\|H\| \leq  T(M,\lambda)$, we have
$$
\# \Fix_\lambda^{\text{\rm trn}}(\psi) \neq \emptyset
$$
\item Provided $\|H\| < T(M,\lambda)$ and $H$ is nondegenerate, we have
$$
\# \Fix_\lambda^{\text{\rm trn}}(\psi) \geq \dim H^*(M; \Z_2).
$$
\end{enumerate}
\end{thm}
\begin{rem}
According to a recent arXiv posting by Cant and Hedicke \cite[Theorem 1]{cant-hedicke}, 
this result is sharp: They 
construct a contact isotopy of $S^{2n-1}$, $n \geq 2$, with the standard contact structure, whose time-one
maps have no translated points relative the standard contact form $\alpha$ thereof, and which satisfy
$$
T(M,\alpha) < \|H|\| \leq T(M,\alpha) + \epsilon, \quad H: = \Dev(t \mapsto \psi_t)
$$
for any $\epsilon > 0$. While the statements of  \cite[Conjecture 31]{shelukhin:contactomorphism}, \cite[Theorem 1]{cant-hedicke} are made differently,
they are equivalent to the inequality given in Theorem \ref{thm:translated-intro}. 
(See the remark made in \cite[Section 1.2]{cant-hedicke}
for a clarification of  some difference in the conventions used
in \cite{albers-frauen} and \cite{shelukhin:contactomorphism}.)
\end{rem}

We now provide this theorem with some general perspective in terms of quantitative contact dynamics
recalling from \cite{oh:entanglement1}.

\begin{defn}[Definition 24 \cite{shelukhin:contactomorphism}] For given $\psi \in \Cont_0(M,\xi)$, we define
$$
\|\psi\|^{\text{\rm osc}}_\lambda = \inf_{H \mapsto \psi} \|H\|.
$$
\end{defn}

\begin{defn}[Definition 1.24 \cite{oh:entanglement1}] Let $(M,\xi)$ be a contact manifold, and let
$S_0, \, S_1$ of compact subsets $(M,\xi)$.
\begin{enumerate}
\item  We define
\be\label{eq:lambda-untangling-energy-intro}
e_\lambda^{\text{\rm trn}}(S_0, S_1) : = \inf_H\{ \|H\| ~|~ \psi_H^1(S_0) \cap Z_{S_1} = \emptyset\}.
\ee
We put $e_\lambda^{\text{\rm trn}}(S_0, S_1) = \infty$ if $\psi_H^1(S_0) \cap Z_{S_1} \neq \emptyset$ for
all $H$. We call $e_\lambda^{\text{\rm trn}}(S_0, S_1)$ the \emph{$\lambda$-untangling energy} of $S_0$ from $S_1$
or just of the pair $(S_0,S_1)$.
\item We put
\be\label{eq:Reeb-untangling-energy-intro}
e^{\text{\rm trn}}(S_0,S_1) = \sup_{\lambda\in \CC(\xi)} e_\lambda^{\text{\rm trn}}(S_0, S_1).
\ee
We call $e^{\text{\rm trn}}(S_0,S_1)$ the \emph{Reeb-untangling energy} of $S_0$ from $S_1$ on $(M,\xi)$.
\end{enumerate}
\end{defn}
We mention that the quantity $e_\lambda^{\text{\rm trn}}(S_0, S_1)$ is \emph{not} symmetric, i.e.,
$$
e_\lambda^{\text{\rm trn}}(S_0, S_1) \neq e_\lambda^{\text{\rm trn}}(S_1, S_0)
$$
in general.

The following dynamical invariant of contact form $\lambda$ is considered in
\cite{oh:entanglement1}. (See also \cite{shelukhin:contactomorphism} for various
invariants of contactomorphisms of similar types.)

\begin{defn}[Untanglement energy of $\lambda$] Let $(M,\xi)$ be a contact manifold.
For each contact form $\lambda$ of $\xi$, we define
\be\label{eq:e-tra-psi}
e_\lambda^{\text{\rm trn}}(M,\xi): = \inf_{H \mapsto \psi}
\{\|\psi\|_\lambda^{\text{\rm osc}} \mid \Fix_\lambda^{\text{\rm trn}}(\psi) = \emptyset\}.
\ee
We call $e_\lambda^{\text{\rm trn}}(M,\xi)$ the \emph{untanglement energy} of the contact form
$\lambda$.
\end{defn}

In this perspective, an immediate corollary of Main Theorem is the following lower bound on
the $e_\lambda^{\text{\rm trn}}(M,\xi)$.
\begin{cor}
$$
e_\lambda^{\text{\rm trn}}(M,\xi) \geq T(M,\lambda).
$$
\end{cor}

There have been various partial results concerning Sandon-Shelukhin's conjecture
for the strongly fillable case for which we refer readers to  \cite{albers-fuchs-merry} and
\cite{shelukhin:contactomorphism},  and \cite{albers-merry}
for the general compact contact manifold but with no precise estimate on the upper bound.
These results rely on the work on leaf-wise intersection points from \cite{albers-frauen}
which uses the machinery of Rabinowitz Floer homology.

\begin{rem} In his arXiv posting \cite{cant}, Cant produced a contactomorphism isotopic
to the identity on the contact sphere $S^{2n+1}$ that has no translated points and hence disproves
Sandon's conjecture on the Arnold-type conjecture on the translated points. He also showed that certain 
contact manifold carries a contactomorphism that has no translated point at all.  It will be a very
interesting question to study what the implication of this untanglement means
in the context of contact dynamics e.g., in the context of thermodynamics
\cite{lim-oh:thermodynamics}.
\end{rem}

\subsection{Contact diffeomorphisms and their Legendrianizations}

The main geometric operation  we employ in the present paper is
the \emph{Legendrianization} of contactomorphism $\psi$. Such an operation was
used by Chekanov \cite{chekanov:generating},  Bhupal \cite{bhupal} and
Sandon \cite{sandon:translated}. It has been also utilized to give a 
locally contractible topology on the group of contactomorphisms by Lychagin \cite{lychagin}, 
and in the study of simpleness question of group of contactomorphisms by
Tsuboi \cite{tsuboi}, Rybicki \cite{rybicki} and the present author \cite{oh:simplicity}. 

For a given contact manifold $(Q,\xi)$ equipped with a contact form $\lambda$, the product $M_Q = Q \times Q \times \R$
carries a natural contact structure given by
\be\label{eq:tildeCA}
\Xi : = \ker {\mathscr A},  \quad {\mathscr A} = -e^{\eta}\pi_1^*\lambda + \pi_2^*\lambda
\ee
associated to each given contact form $\lambda$ of $(Q,\xi)$ which was used in \cite{bhupal}
(modulo the sign).  An upshot of the present paper is to utilize some hidden involutive 
symmetry of this contact structure $\Xi$. For this purpose, it is natural to take the contact form
$$
\CA = e^{-\frac{\eta}{2}} \mathscr A = -e^{\frac{\eta}{2}}\pi_1^*\lambda +
e^{-\frac{\eta}{2}}  \pi_2^*\lambda
$$
for which the anti-symmetry is apparent under the involution $\iota(x,y,\eta): = (y,x-\eta)$.

\begin{defn}[$\Gamma_{\psi}$] Let $\psi: (Q,\xi) \to (Q,\xi)$ be a contactomorphism. 
For each given contact form $\lambda$, 
we define the \emph{Legendrianization} of $\psi$ to be the submanifold
\be\label{eq:contact-graph}
\Gamma_\psi = \Gamma_\psi^\lambda: = \{(\psi(q),q, g_\psi^\lambda(q)) \in Q \times Q \times \R \mid 
q \in Q \}.
\ee
\end{defn}
Here we would like to recall readers 
 that the definitions of Legendrian submanifold and contactomorphisms do not
 on the choice of contact forms \emph{depending only on  the contact
structure}. Therefore for the purpose of proving the main result of
the present paper, the contact form can be freely chosen 
\emph{as long as the contact form $\lambda$ on $Q$ and the contact structure 
$\Xi$ on $M_Q$ are fixed.} In our case the choice of $\CA$ is the one
we use because it is compatible with the involutive symmetry of the contact structure $\Xi$.

The following lemma  provides a link between the intersection set $\Gamma_\psi \cap Z_{\Gamma_\id}$
in $M_Q$ and the translated points of a contactomorphism $\psi$ on $Q$ which
converts the translated point problem into that of translated intersection problem
\emph{with quantitative comparison}.

\begin{lem}[Corollary \ref{cor:link}] \label{lem:link-intro}
$$
e^{\text{\rm trn}}_{\CA}(\Gamma_{\text{\rm id}},\Gamma_{\text{\rm id}}) \geq
T(M_Q,\CA;\Gamma_{\text{\rm id}})
\Longrightarrow e^{\text{\rm trn}}_\lambda(Q,\xi) \geq T(Q,\lambda).
$$
\end{lem}

The following general result on the existence of Reeb chords 
was proved for the two-component Legendrian link
$(\psi(R),R)$ for any compact Legendrian submanifold and contactomorphism of
a \emph{tame} contact manifold $(M,\lambda)$ in \cite{oh:entanglement1}.

\begin{thm}[Theorems 1.13 \cite{oh:entanglement1}]\label{thm:shelukhin-intro}
Let $(M,\xi)$ be a contact manifold equipped with a tame contact form $\lambda$.
Let $\psi \in \Cont_0(M,\xi)$ and
consider any Hamiltonian $H = H(t,x)$ with $H \mapsto \psi$.
Assume $R$ is any compact
Legendrian submanifold of $(M,\xi)$. Then the following hold:
\begin{enumerate}
\item Provided $\|H\| \leq T_\lambda(M,R)$, we have
$$
\#\mathfrak{Reeb}(\psi(R),R) \neq \emptyset.
$$
\item Provided $\|H\| < T_\lambda(M,R)$
and $\psi= \psi_H^1$ is nondegenerate to $(M,R)$, then
$$
\#\mathfrak{Reeb}(\psi(R),R) \geq \dim H^*(R; \Z_2)
$$
\end{enumerate}
\end{thm}
This theorem is known to be optimal on tame contact manifold
in general. (See \cite[Example 1.4 \& Lemma 1.6]{rizell-sullivan:algebra}. That paper
also contains relevant results on the \emph{contactification} $P \times \R$ of an exact symplectic
manifold $P$ which utilizes the machinery of Chekanov-Eliashberg DGA.)

\subsection{Contact product $M_Q$ and anti-contact involution}

Furthermore Theorem \ref{thm:shelukhin-intro}, together with Lemma \ref{lem:link-intro},
 would have already proved Shelukhin's conjecture 
stated in \cite[Conjecture 31]{shelukhin:contactomorphism} \emph{provided}
\begin{enumerate}
\item the contact product
 $M_Q = Q \times Q \times \R$ were tame:The notion of tame contact manifolds
 is introduced by the present author in \cite{oh:entanglement1}. 
 See the beginning of Part \ref{part:shelukhin-conjecture} for the recall of the definition.
 \item the Hamiltonian $H$ generating $\psi = \psi_H^1$ can be lifted to one $\H$ which
 generates the isotopy $\Gamma_{\psi_H^t}$ on $M_Q$ so that $\|H\| = \|\H\|$.
 \end{enumerate}

So the main task, in the presence of the aforementioned theorem, is to establish these two
results in the present paper.

For example, the \emph{noncompact} contact manifold  $M_Q = Q \times Q \times \R$ 
of our main interest in the present paper is  tame. 

Here enters the following crucial observation:

\begin{observation}\label{observ:iota} $\Gamma_{\text{\rm id}}$ is the fixed point set of
anti-contact involution $\iota: M_Q \to M_Q$ defined by
\be\label{eq:iota}
\iota(x,y,\eta) = (y,x, -\eta).
\ee
\end{observation}

To  exploit this involutive symmetry
of the contact structure, we need to consider the contact form $\CA$ given by
$$
\CA =  - e^{\frac{\eta}{2}}\pi_1^*\lambda + e^{-\frac{\eta}{2}} \pi_2^*\lambda
$$
which carries the involutive anti-symmetry $\iota^*\CA = -\CA$. We mention that this contact form
is different from $\mathscr A = - e^\eta\pi_1^*\lambda +  \pi_2^*\lambda = e^{\frac{\eta}{2}}\CA$
employed in \cite{bhupal} (with different sign convention):  
the latter is geometrically natural to consider if one would like to
apply the standard method of pseudoholomorphic curves on symplectization.
(See Remark \ref{rem:AA-versus-CA} in the beginning of Part \ref{part:contact-product} for the reason why.)

We use the form $\CA$ in the present paper  which  is more `symmetric' with respect
to the two factors of $Q$'s. The involutive symmetry of the associated contact structure $\Xi$
will be in fact apparent in the more functorial construction of contact
product  introduced in the present paper. See Section \ref{sec:contact-product} or below.

Then we have the following tameness of $(M_Q, \CA)$.

\begin{prop}[Proposition \ref{prop:tame}] \label{prop:tame-intro}
The contact manifold $(M_Q,\CA)$ is tame in the sense of
Definition \ref{defn:tame}.
\end{prop}

Then exploiting this involutive symmetry, we will apply the $\Z_2$-equivariant 
version of the contact instanton equation, which also enables us to 
lift the Hamiltonian flow of $H$ on $Q$ to that of a lifted Hamiltonian $\H'$ on $M_Q$ 
such that  $\|H\| = \|\H'\|$. (See \eqref{eq:HH'}.) 
 This practice requires us to carefully orchestrate 
the choices of the family of CR almost complex
structures and of (lifted) Hamiltonians for the relevant perturbed contact instantons. 
(See  Section \ref{sec:lifted-Hamiltonian} for the details.)

\subsection{Contact-product and perturbed contact instantons}

Now a brief description of the second ingredient above is in order.
We have one-to-one correspondence between $\Fix^{\text{\rm trn}}(\psi)$ and the set
$\mathfrak{Reeb}(\Gamma_{\psi},\Gamma_{\text{\rm id}})$ of Reeb chords between 
$\Gamma_{\psi}$ and $\Gamma_{\text{\rm id}}$. (We refer readers to Section \ref{sec:contact-product} for the details
of this correspondence.)

\begin{rem} In fact, there is a more functorial description of this product operation, which
we call the \emph{contact product}, denoted by
$$
Q_1 \star Q_2 = (Q_1,\xi_1) \star (Q_2,\xi_2)
$$
and explained in Section \ref{sec:contact-product}.  This functorial construction will be systematically
further developed elsewhere \cite{oh-park} with Junhyuk Park.
\end{rem}

We then  consider some  lifted Hamiltonian $\H = \H(t, x,y,\eta)$ on $M_Q$ 
whose Hamiltonian vector field satisfies
\be\label{eq:liftedH}
\iota^*X_\H = -X_\H.
\ee
(See Section \ref{sec:lifted-Hamiltonian} for the explicit construction of such $\H$.)

Let ${\CA}$ and $\H$ be as above.
Denote ${\bf x} = (x,y,\eta)$. Then we consider the associated \emph{curvature-free}
2-parameter family $\{\H^s\}$ defined by
\be\label{eq:2parameter-family}
\H^s(t,{\bf x})= \Dev_\lambda\left(s \mapsto \psi_{\H}^{st}({\bf x})\right)
\ee
and
$$
\H_K(\tau,t,{\bf x}) = \Dev_\lambda\left(t \mapsto \psi_{\H}^{\chi_K(\tau)t}({\bf x})\right)
$$
for the standard elongation function $\chi:\R \to [0,1]$ used as in \cite{oh:entanglement1}:
Then we consider the family of cut-off functions to elongate the about family over
$[0,1]^2$ to $\R \times [0,1]$, which is the same kind as the ones used in 
\cite{oh:mrl,oh:entanglement1}. 
We write
$$
{\mathbb G}_K(\tau,t,{\bf x}) = \Dev_{\CA}\left(\tau \mapsto \psi_{\H}^{\chi_K(\tau)t}({\bf x})\right)
$$
as in \cite{oh:entanglement1}.

Now we consider the following perturbed contact instanton equation associated to $(M_Q,\Gamma_{\text{\rm id}})$,
$\widetilde J$ and $\H^s$
\be\label{eq:K-intro}
\begin{cases}
\left(dU - X_{\H_K}(U)\, dt - X_{\mathbb G_K}(U)\, ds \right)^{\pi(0,1)} = 0, \\
d\left(e^{\widetilde g_K(U)}(U^*\Lambda + U^*{\H}_K dt
+ U^*{\mathbb G}_K\, d\tau) \circ j\right) = 0,\\
U(\tau,0), \quad U(\tau,1)  \in \Gamma_{\text{\rm id}}
\end{cases}
\ee
for a map $U: \R \times [0,1] \to Q \times Q\times \R$. We define its conjugate $\widetilde U$ by
\be\label{eq:tildeU-intro}
\widetilde U (\tau,t) :=  \iota(U(\tau,1-t)).
\ee
The tameness of the lifted almost complex structure $\widetilde J$ on $M_Q$ 
(Definition \ref{defn:tildeJ} and Proposition \ref{prop:tame}) immediately gives rise to 
the aforementioned uniform $C^0$ estimates.
\begin{cor}\label{cor:C0bound-intro}
 There exists a constant $C = C(\{H^s\}, \{\widetilde J^s\}, \{g_{H^s}\}\}) > 0$ such that
$$
\Image U \subset \{|\eta| < C\}
$$
for all  solutions $U$ of \eqref{eq:K-intro}.
\end{cor}

We then consider the associated parameterized moduli space
\be\label{eq:symmetric-CM}
\CM^{\text{\rm para}}_{[0,K]}(M_Q, \Gamma_{\text{\rm id}};\widetilde J, \H)
\ee
for $0 \leq K < \infty$
consisting of the maps $U$ satisfying \eqref{eq:K-intro}
with $J, \, H$ replaced by $\widetilde J,\, \H$ and $R$ by $\Gamma_{\text{\rm id}}$ from Theorem \ref{thm:shelukhin-intro}, respectively.

Main Theorem is a consequence of the following result on the existence of Reeb chords for the link
$(R_0,R_1)$ with $R_0 = \psi_\H^1(R_1)$ for a suitable choice of $R_1$. (See Section \ref{sec:lifted-Hamiltonian}.)

\begin{thm}\label{thm:dual-pair} 
Let $(Q,\xi)$ be a compact contact manifold equipped with
a contact form $\lambda$.
Consider $(M_Q,\Xi)$  equipped with the contact form
$\CA $.
Let $H = H(t,x)$ be any Hamiltonian $\psi_H^1 = \psi$ on $Q$.
Then the following hold:
\begin{enumerate}
\item Provided $\|H\| \leq T_\lambda(M,R)$, we have
$$
\#\mathfrak{Reeb}(\Gamma_{\psi_H^1}, \Gamma_\id;\CA) \neq \emptyset.
$$
\item Provided $\|H\| < T_\lambda(M,R)$
and $\psi= \psi_H^1$ is nondegenerate to $M$, then
$$
\#\mathfrak{Reeb}(\psi_H^1(\Gamma_\id),\Gamma_\id;\CA) \geq \dim H^*(Q; \Z_2)
$$
\end{enumerate}
\end{thm}
As mentioned before, this would be a special case of Theorem \ref{thm:shelukhin-intro}, if
the contact product $(M_Q,\Xi)$ were tame. Because of non-tameness thereof,
this theorem is independent from it and is some extension thereof to a non-tame
contact manifold $Q \times Q \times \R$.

\subsection{Discussion and the relationship with other approaches}

In the present paper, the $\Z_2$ symmetry mainly enters in the proof of tameness of $M_Q$. 
Once this tameness is achieved the whole
methodology and the details of the proof of Theorem \ref{thm:shelukhin-intro} in 
\cite{oh:entanglement1} can be applied, however,
with one important catch: Unlike the contact form $\mathscr A$, it is rather nontrivial to 
lift the given contact isotopy $\psi_H^t$ of $(Q,\lambda)$ to one $\Psi_t$ on the product $(M_\Q,\CA)$
in the way that the following two requirements hold: 
\begin{enumerate}
\item $\|\Dev_\CA(t \mapsto \Psi_t)\| = \|H\|$,
\item There is a one-to-one correspondence between the set of translated points of $\psi_H^1$ and
the set of Reeb chords between $\Gamma_{\psi_H^1}$ and  $\Gamma_\id$.
\end{enumerate}

We would like to emphasize that to the best knowledge of ours
there is no counterpart of these practices in other existing 
approaches such as the Rabinowitz Floer homology, Eliashberg-Chekanov DGA or even more general
machinery of SFT.  (See Remark \ref{rem:amalgamation} for some
relevant illustration.) There is at least one apparent reason for this:
There has been no action functional in other approaches 
that emulates the optimality of the calculations performed in the present paper:
Our calculations  strongly rely on the
\emph{amalgamation} of  the moduli theory of 
\emph{(perturbed) contact instantons with Legendrian boundary condition} and the \emph{contact Hamiltonian calculus}. 
Such an amalgamation thereof was first introduced
in \cite{oh:entanglement1,oh-yso:spectral} and much amplified in the present paper
by incorporating the $\Z_2$-invariant contact products. 
We would like to compare this amalgamation with that of \emph{Floer's Hamiltonian-perturbed
pseudoholomorphic curve equations} and the \emph{symplectic Hamiltonian calculus}. The former calculus is much
more nontrivial and interesting than the latter because of the  role of the \emph{conformal exponent}
$g_{(\psi;\lambda)}$
in the calculus (let alone its fundamental role in the general contact dynamics).
This is a new phenomenon in contact dynamics that is not present in the symplectic case. 

As far as the present author sees, this is a nontrivial huddle to overcome if one tries to 
employ the existing approaches of pseudoholomorphic curves, and there seems no obvious way of 
translating the practices done via the contact instantons in the present paper
 into those via the pseudoholomorphic curve approach.

\bigskip

\noindent{\bf Acknowledgement:} We would like to thank Dylan Cant for pointing out the
presence of the example \cite[Example 1.4]{rizell-sullivan:algebra} which shows the optimality of
Theorem \ref{thm:shelukhin-intro} for general compact Legendrian submanifold $R$. 
 We thank him also for attracting our attention to the recent
arXiv posting \cite{cant-hedicke} which clarifies the statement of Shelukhin's conjecture.
We thank the unknown referees for their hard work of going through the details of the 
paper and pointing out errors related to the definition of tameness of contact manifold and
to the choice of 2-parameter family of Hamiltonian $\H$ and other numerous small errors
in the exposition. We also thank them for many kind and helpful suggestions to improve
the exposition of the paper.
\bigskip

\noindent{\bf Convention and Notations:}

\medskip

\begin{itemize}
\item {(Contact Hamiltonian)} We define the contact $\lambda$-Hamiltonian of a contact
vector field $X$ to be $- \lambda(X) =: H$.
\item{ (Contact Hamiltonian flow)} We denote by $X_H = X_{(H;\lambda)}$ the
$\lambda$-Hamiltonian vector field of $H$ and by
 $\psi_H^t= \psi_{H;\lambda}^t$ its flow.
\item {(Reeb vector field)} We denote by $R_\lambda$ the Reeb vector field associated to $\lambda$
and its flow by $\phi_{R_\lambda}^t$.
\item {(Conformal factor)} For any contactomorphism $\psi$ with $\psi^*\lambda = f_\psi \lambda$, we call
$f_\psi = f_\psi^\lambda$ the \emph{conformal factor} of $\psi$ with respect to $\lambda$.
\item {(Conformal exponent)}
When $\psi$ is orientation-preserving for which the conformal factor $f_\psi$ is positive, we call its logarithm
$g_\psi^\lambda: = \log f_\psi^\lambda$ of the conformal factor
$f_\psi^\lambda$ the \emph{conformal exponent} of contactomorphism $\psi$
with respect to $\lambda$. Hence $\psi^*\lambda = e^{g_\psi^\lambda}$ with $g_\psi = g_\psi^\lambda$.
\item {(Exponent flow)} For given (time-dependent) Hamiltonian $H = H(t,x)$, we denote by $g_H$ the function
$g_{(H;\lambda)} (t,x) := g_{\psi_H}^\lambda(t,x) = g_{\psi_H^t}^\lambda(x)$ and call it the
$\lambda$-exponent flow of $H$.
\item {(Time-reversal Hamiltonian)}  For given $H = H(t,x)$, we denote by 
$\widetilde H$ the Hamiltonian defined by $\widetilde H(t,x) = - H(1-t,x)$. 
Its Hamiltonian flow is given by
\be\label{eq:time-reversal-flow} 
\psi_{\widetilde H}^t = \psi_H^{1-t}\circ (\psi_H^1)^{-1}
\ee
issued at $\psi_{\widetilde H}^0= id$.
\item {(Functorial contact product)} $Q \star Q: = S^{\text{\rm can}} \times 
S^{\text{\rm can}}/\sim$
\item  $M_Q: = Q \times Q \times \R$.
\end{itemize}

\section{Basics of contact Hamiltonian calculus}

In this section, we recall some basic Hamiltonian calculus that will play key roles for
the purpose of the present paper. These are mainly from 
\cite[Section 2]{oh:contacton-Legendrian-bdy} with the sign conventions for the
calculus and from \cite{mueller-spaeth-I} basic formulae for the Hamiltonian of the product of contact isotopy.

\subsection{Contact isotopy, conformal exponent and contact Hamiltonian}

\begin{lem}\label{lem:LXtlambda} Let $X$ be a contact vector field and $H = -\lambda(X)$ be the associated
Hamiltonian. Then
\be\label{eq:dH}
dH = X_H \rfloor d\lambda + R_\lambda[H] \lambda
\ee
and $\CL_X \lambda = R_\lambda[H] \lambda$.
\end{lem}
Conversely, when a function $H$ is given, the associated Hamiltonian vector field $X_H$ is 
uniquely determined by
\be\label{eq:defining-XH}
\begin{cases}
X \rfloor \lambda = -H \\
X \rfloor d\lambda = dH - R_\lambda[H]\lambda
\end{cases}
\ee

Let $\psi_t$ be a contact isotopy of $(M, \xi = \ker \lambda)$ with $\psi_t^*\lambda = e^{g_t}\lambda$
and let $X_t$ be the time-dependent vector field generating the isotopy. Let
$H: [0,1] \times M \to \R$ be the associated time-dependent contact Hamiltonian
$H_t = - \lambda(X_t)$. One natural question to ask is explicitly what the relationship
between the Hamiltonian $H = H(t,x)$ and the conformal exponent
\be\label{eq:gH}
g_{(H;\lambda)} = g_{(H;\lambda)}(t,x): = g_{\psi_H^t}^\lambda(x).
\ee
The following lemma is a straightforward consequence
of the identity $(\phi\psi)^*\lambda = \psi^*\phi^*\lambda$.
\begin{lem}\label{lem:coboundary} Let $\lambda$ be given and denote by $g_\psi$ the function $g$
appearing above associated to a contactomorphism $\psi$. Then
\begin{enumerate}
\item $g_{\phi\psi} = g_\phi\circ \psi + g_\psi$ for all $\phi, \, \psi \in \Cont(M,\xi)$,
\item $g_{\psi^{-1}} = - g_\psi \circ \psi^{-1}$ for all $\psi \in \Cont(M,\xi)$.
\end{enumerate}
\end{lem}
From now on, when $\lambda$ is fixed or is clear, we will omit it from notations, e.g., by
just writing $g_{H}$ for $g_{(H;\lambda)}$.

The following formula provides an explicit key relationship between the contact Hamiltonians
and the conformal exponents of a given contact isotopy.

\begin{prop}[Proposition 2.7 \cite{oh:contacton-Legendrian-bdy}]\label{prop:gH}
Let $\Psi=\{\psi_t\}$ be a contact isotopy of $(M, \xi = \ker \lambda)$ with $\psi_t^*\lambda = e^{g_t}\lambda$
with $\Dev_\lambda(\Psi) = H$. We have $g_{(H;\lambda)}(t,x) = g_t(x)$. Then
\be\label{eq:dgdt}
\frac{\del g_{(H;\lambda)}}{\del t}(t,x) = -  R_\lambda[H](t,\psi_t(x)).
\ee
In particular, if $\psi_0 = id$,
\be\label{eq:gp}
g_{(H;\lambda)}(t,x) = \int_0^t - R_\lambda[H] (u,\psi_u(x))\, du.
\ee
\end{prop}

\subsection{Developing Hamiltonians and their calculus}

In this subsection, we recall the contact counterparts of  the developing map 
and the tangent map partially following the exposition given in 
\cite{oh:hameo1,oh:hameo2} and \cite{mueller-spaeth-I} with slight amplification thereof.

Similarly as in the symplectic case \cite{oh:hameo1}, 
we first name the parametric assignment $\psi_t \to -\lambda(X_t)$
of $\lambda$-Hamiltonians the \emph{$\lambda$-developing map} $\Dev_\lambda$, the analog to
the  developing map introduced in \cite{oh:hameo1,oh:hameo2} in the context of symplectic
geometry.

We will just call them the developing map,
when there is no need to highlight the $\lambda$-dependence of the contact Hamiltonian.

\begin{defn}[Developing map]\label{defn:Dev-lambda} 
Let $T \in \R$  be given. Denote by
$$
\CP_0([0,T], \Cont(M,\xi))
$$
the set of contact isotopies $\Psi= \{\psi_t\}_{t \in [0,T]}$ with $\psi_0 = id$.
We define the $\lambda$-developing map
$
\Dev_\lambda: \CP([0,T], \Cont(M,\xi)) \to C^\infty([0,T] \times M)
$
by the timewise assignment of Hamiltonians
$$
\Dev_\lambda(\Psi) : = -\lambda(X), \quad \lambda(X)(t,x) := \lambda(X_t(x)).
$$
\end{defn}
We also state the following as a part of contact Hamiltonian calculus, whose proof is
a straightforward calculation.

We also state the following as a part of contact Hamiltonian calculus, 
which follows from a straightforward calculation. We refer  readers to
to \cite{mueller-spaeth-I} for detailed proofs.

\begin{enumerate}
\item  Let $H = \Dev_\lambda(\Psi)$. 
$\Dev_\lambda(\widetilde \Psi)(t,x) = - H(1-t,x)$ for the time-reversal isotopy $t\mapsto \psi^{1-t}$,
\item 
$
\Dev_\lambda(t\mapsto \psi_H^t \psi_{H'}^t) = H_t + e^{g_{\psi_H^t}}\, F_t (\psi_H^t)^{-1}
$
when $H = \Dev_\lambda(\Psi)$ and $H' = \Dev_\lambda(\Psi')$,
\item $\Dev(t \mapsto \psi^{-1} \psi_H^t \psi) = e^{-g_\psi} H_t \circ \psi$ for the conjugation $t \mapsto \psi^{-1} \psi_H^t \psi$.
\end{enumerate}
 One important task in 
quantitative contact Hamiltonian dynamics is to optimally amalgamate the contribution from
the Hamiltonian and that of its associated conformal exponent functions.

\section{Perturbed contact instantons and their subsequence convergence}

Now we recall the \emph{perturbed contact instanton equation}
from \cite{oh:contacton-Legendrian-bdy} and \cite{oh:perturbed-contacton}
\be\label{eq:perturbed-contacton-bdy}
\begin{cases}
(du - X_H \otimes dt)^{\pi(0,1)} = 0, \\
 d(e^{g_{(H;u)}}u^*(\lambda + H \, dt)\circ j) = 0
 \end{cases}
\ee
together with the boundary condition
\be\label{eq:Legendrian-bdy-condition}
u(\tau,0) \in R, \quad u(\tau,1) \in R
\ee
when the domain surface  $\dot \Sigma$ is $ \R \times [0,1]$.  Here the function $g_{H,u}$ is
the function
\be\label{eq:gHu}
g_{(H;u)}(\tau,t): = g_{(\psi_{\widetilde H}^t)^{-1}}(u(\tau,t)) = g_{\psi_H^1 (\psi_H^{1-t})^{-1}}(u(\tau,t)),
\ee
as used in \cite{oh:contacton-Legendrian-bdy,oh:entanglement1}.

A straightforward standard calculation also gives rise to the following gauge-transformed
equation. (See \eqref{eq:ubarK} for the definition of $\overline u$.)

\begin{lem}[Lemma 6.7 \cite{oh:entanglement1}]\label{lem:Ham-Reeb} For given $J_t$, consider $J'$ defined as
above. We equip $(\Sigma,j)$ a K\"ahler metric $h$. 
Suppose $u$ satisfies \eqref{eq:perturbed-contacton-bdy}
with respect to $J_t$. Then its gauge transform $\overline u$ satisfies
\be\label{eq:contacton-bdy-J0}
\begin{cases}
\delbar^\pi_{J'} \overline u = 0, \quad d(\overline u^*\lambda \circ j) = 0 \\
\overline u(\tau,0) \in \psi_H^1(R), \, \overline u(\tau,1) \in R
\end{cases}
\ee
for $J$. The converse also holds. And $J' = J'(\tau,t)$ satisfies
$
J'(\tau,t) \equiv J_0
$
for $|\tau|$ sufficiently large.
\end{lem}

Imitating the practice exercised by the author in Lagrangian Floer theory \cite{oh:jdg,oh:cag}, we 
call \eqref{eq:perturbed-contacton-bdy} the \emph{dynamical CI equation} and \eqref{eq:contacton-bdy-J0}
the \emph{geometric CI equation}.
\begin{rem} One difference about the gauge transformation from \cite{oh:entanglement1} over the time interval 
$[0,\frac12]$ is that
we use the transformation arising by applying the one arising from $t \mapsto \psi_{\widetilde H}^t$ here while
we used the flow $t \mapsto (\psi_H^t)^{-1}\psi_H^1$ in \cite{oh:entanglement1}:  This  makes
the definition of exponent function $g_{(H;u)}$ above different from that of \cite{oh:contacton-Legendrian-bdy,oh:entanglement1}.
\end{rem}

The following subsequence convergence result is proved in
\cite[Theorem 8.5]{oh:perturbed-contacton}, \cite[Theorem 5.5]{oh-yso:index} which exhibits the relationship
between the asymptotic limit of a finite energy solution of \eqref{eq:perturbed-contacton-bdy} and
the translated points of $\psi_H^1$.

\begin{thm}[Subsequence Convergence]\label{thm:perturbed-subsequence}
Let $u:[0, \infty)\times [0,1]\to M$ satisfy the contact instanton equations
\eqref{eq:perturbed-contacton-bdy} and have finite energy.

Then for any sequence $s_k\to \infty$, there exists a subsequence, still denoted by $s_k$, and a
massless instanton $u_\infty(\tau,t)$ (i.e., $E^\pi(u_\infty) = 0$)
on the cylinder $\R \times [0,1]$  that satisfies the following:
\begin{enumerate}
\item $\delbar_H^\pi u_\infty = 0$ and
$$
\lim_{k\to \infty}u(s_k + \tau, t) = u_\infty(\tau,t)
$$
in the $C^l(K \times [0,1], M)$ sense for any $l$, where $K\subset [0,\infty)$ is an arbitrary compact set.
\item $u_\infty$ has vanishing asymptotic charge $Q = 0$ and satisfies
$$
u_\infty(\tau,t)= \psi_H^t\phi_{R_\lambda^{Tt}}(x_0)
$$
for some $x_0 \in R_0$.
\end{enumerate}
\end{thm}
We note that at $t=1$, $u_\infty(\tau,1) = \psi_H^1\phi_{R_\lambda^{T}}(x_0)$. Therefore
 then the curve $\gamma$ defined by
\be\label{eq:gamma-x0}
\gamma(t) = \psi_H^t \phi_{R_\lambda^{tT}}(x_0)
\ee
satisfies $\gamma(0) \in R_1,\, \, \gamma(1) \in R_0$ and appears as the asymptotic limit of
a finite energy solution $u$ of \eqref{eq:perturbed-contacton-bdy}.

Motivated by this subsequence result, we introduce the following translated version of Hamiltonian chords.

\begin{defn}[Translated Hamiltonian chords] \label{defn:trans-Ham-chords}
Let $(R_0,R_1)$ be a 2-component Legendrian link of $(M,\lambda)$.
\begin{enumerate}
\item
We call a curve $\gamma$ of the form \eqref{eq:gamma-x0} a \emph{translated Hamiltonian chord} from $R_0$ to $R_1$ with $x_0 \in R_0$.
We denote by
$$
\mathfrak X^{\text{\rm trn}}((R_0,R_1);H)
$$
the set thereof.
\item
We call the intersection $(\psi_H^1)^{-1}(Z_{R_0})\cap R_1$ the set of \emph{$\lambda$-translated Hamiltonian intersection points}.
\end{enumerate}
\end{defn}

The following obvious lemma shows the relationship between the set of translated Hamiltonian intersection points
of $(\psi_H^1)^{-1}(Z_{R_1})$ and $R_0$
and the intersection $\psi_H^1(R_1) \cap Z_{R_0}$. 

\begin{lem}\label{prop:trnhamchords-trnReebchords}
Let $R_1 = \psi_H^1(R_0)$. Then we have one-one correspondences
$$
\xymatrix{
(\psi_H^1)^{-1}(Z_{R_1})\cap R_0 \ar[d]\ar[r] & \ar[l]\psi_H^1(R_0) \cap Z_{R_1} \ar[d]\\
\mathfrak X^{\text{\rm trn}}((R_0,R_1);H) \ar[u] \ar[r] & \ar[l] \mathfrak{Reeb}(R_0, R_1) \ar[u].
}
$$
\end{lem}

\begin{rem}
The set $\mathfrak X^{\text{\rm trn}}((R_0,R_1);H)$ will play the role of generators of the
Floer homology associated to \eqref{eq:perturbed-contacton-bdy} and $\mathfrak{Reeb}(R_0, R_1)$
of the Floer homology of its gauge transform \eqref{eq:contacton-bdy-J0} below.
\end{rem}

\part{Contact product and its Hamiltonian geometry}
\label{part:contact-product}

In this part, we provide a functorial construction of \emph{product operation}
of contact manifolds and
its $\Z_2$-symmetry of anti-contact involution. We then develop the relevant
$\Z_2$-equivariant contact Hamiltonian calculus which will play a fundamental role in
our proof of Shelukhin's conjecture in combination with the analysis of contact
instantons.

\begin{rem}\label{rem:AA-versus-CA} The contact form $\mathscr A$ used in \cite{bhupal}
is closely related to
the symplectization $(Q \times \R, d(e^\eta \mathscr A))$ in that the natural map
$$
S(Q \times Q \times \R) \to S(Q \times SQ); \quad (x,y,\eta, s) \mapsto (x, (y,\eta),s)
$$
is a symplectomorphism if one equip the domain and the codomain with the following natural
symplectic forms
$$
\Omega_1 = d(e^s \mathscr A), \quad \left(\text{\rm or}\quad \Omega_2 = d(e^s(-\pi_1^*\lambda +  e^\eta \pi^*\lambda))\right)
$$
where $\pi: SQ \to Q$ is the canonical projection. However, as mentioned in the introduction,
because the contact manifold $(Q \times Q \times \R, \Xi)$ is not tame on the region $\{\eta < 0\}$,
the contact form $\mathscr A$ is not suitable as it is for our purpose. 
\end{rem}

In the present paper, we pursue for the argument of involutive symmetry in the contact geometry 
to overcome the aforementioned non-tameness of the region $\{\eta < 0\}$,
\emph{without taking the symplectization.}

\section{Contact product and Legendrianization}
\label{sec:contact-product}

In this section we introduce a functorial construction of \emph{product operation}
of contact manifolds $(Q_1,\xi_1)$ and $(Q_2,\xi_2)$ which carries a natural $\Z_2$-symmetry
of anti-contact involution when $(Q_1,\xi_1)= (Q_2,\xi_2)$.

\subsection{Definition of contact product}

Let $(M, \theta)$ be any complete Liouville (symplectic) manifold.

We start with the following well-known fact. (See \cite{cieliebak-eliashberg}, \cite{giroux}, for example.)
For readers' convenience, we include some discussion on
the ideal boundary by recalling  standard definitions and properties of the ideal
boundary of a Liouville manifold in Appendix \ref{sec:liouville}, to which we refer
readers for the unexplained terms in the statement.

\begin{prop} Let $(M, \theta)$ be any complete Liouville (symplectic) manifold.
Then the ideal boundary  $\del_\infty M$ consisting of (one-sided) Liouville rays carries a
canonical contact structure. When the Liouville flow admits a cross section,
then the restriction of $\theta$ to the cross section defines a contact structure
that is contactomorphic to $\del_\infty M$.
\end{prop}

Next we recall that  the canonical symplectization of a contact manifold $(Q,\xi)$
is defined by
$$
S^{\text{\rm can}}Q = \{\mathfrak{B} \in T^*Q \setminus \{0\} \mid \mathfrak{B}|_\xi = 0 \}
$$
equipped with its symplectic form
$$
\omega_0 = \sum_{i=1} dq_i \wedge dp_i = - d\theta|_{SQ}
$$
where $\theta$ is the Liouville one-form on $T^*Q$.
(See e.g. \cite[Appendix]{arnold-book} for the definition.)

When $\xi$ is cooriented, $S^{\text{\rm can}}Q$ has two connected components, denoted by
$S_+Q$ and $S_-Q$. Each of them is symplectomorphic to $Q \times \R_+$
equipped with the symplectic form $d(r \lambda)$ provided a contact form
$\lambda$ is chosen. Then $(Q,\xi)$ is naturally contactomorphic to
$(SQ, \ker \theta|_{SQ})/\sim$. When $(Q,\xi)$ is cooriented, we just denote
$S_+ Q = Q \times \R_+$ equipped with $[\pi^*\theta]$.

Now let $(Q_1, \xi_1)$ and $(Q_2,\xi_2)$ be two cooriented contact manifolds.
We consider the product symplectic manifold
$$
(S^{\text{\rm can}}Q_1 \times S^{\text{\rm can}}Q_2, \pi_1^*\omega_0^1 + \pi_2^*\omega_0^2).
$$
Then we have
$$
\pi_1^*\omega_0^1 + \pi_2^*\omega_0^2 = d(-\pi_1^*\theta_1 - \pi_2^*\theta_2).
$$
\begin{lem} The Liouville one-form $-\pi_1^*\theta_1 - \pi_2^*\theta_2$ of the exact symplectic form
$\pi_1^*\omega_0^1 + \pi_2^*\omega_0^2$ has
the Liouville vector field given by
$$
(X_1, X_2)
$$
where $X_i$ is the Liouville vector field  of $S^{\text{\rm can}}Q_i$, $i = 1,\, 2$ which
are the standard Liouville vector fields $p \frac{\del}{\del p}$
of $T^*Q_i$ restricted to $S^{\text{\rm can}}Q_i$.
\end{lem}

\begin{defn}[Contact product] For each given pair of contact manifolds $(Q_i,\xi_i)$, $i = 1, \, 2$,
we define their \emph{contact product}, denoted by $Q_1 \star Q_2$, to be
$$
Q_1 \star Q_2 = S^{\text{\rm can}}Q_1 \times S^{\text{\rm can}}Q_2/\sim
$$
equipped with the aforementioned canonical contact structure defined by
$$
\xi_{Q_1 \star Q_2} = [\ker (-\pi_1^*\theta_1 - \pi_2^*\theta_2)] \subset T(Q_1 \star Q_2)
= T(S^{\text{\rm can}}Q_1 \times S^{\text{\rm can}}Q_2)/\sim.
$$
\end{defn}

\begin{rem} The functorial study of contact products 
of any, not necessarily coorientable, and the calculus of their Legendrian submanifolds
will be systematically further developed in \cite{oh-park} with Junhyuk Park.
\end{rem}

\subsection{Contact product $Q \star Q$ and anti-contact involution}

We now specialize to the case $Q_1 = Q_2 = Q$.
Note that $M_Q$ carries a natural involution
\be\label{eq:involution}
\iota(x,y,\eta) = (y,x,-\eta)
\ee
such that
$$
\Gamma_{\text{\rm id}} = \Fix \iota.
$$
Existence of such an involution is special for $R = \Gamma_{\text{\rm id}}$
which will enable us to prove Main Theorem, Theorem \ref{thm:translated-intro},
by the symmetry argument that is similar to the one exploited by Fukaya-Oh-Ohta-Ono in
\cite{fooo:anti-symplectic}.

We say the involution $\iota$ \emph{anti-contact} in the following sense,
whose proof is omitted.

\begin{lem} Let $\Xi$ be the contact structure on $M_Q$ given by
$$
\Xi: = \ker {\mathscr A}.
$$
Then $\iota^*\Xi = \Xi$. Furthermore the contact form 
$[- \pi_1^*\theta + \pi_2^\theta]=: \Lambda$ also satisfies $\iota^\Lambda = \Lambda$
\end{lem}

 The following is another important 
property of the contact product $Q \star Q$.
\begin{lem} Assume that $(Q,\xi)$ is cooriented equipped with a contact form
$\lambda$. Then the canonical projection 
$$
Q \star Q = S^{\text{\rm can}} Q \times S^{\text{\rm can}} Q/\sim  \, \to Q \times Q
$$
induces a natural section.
\end{lem}
\begin{proof} At each point $(x,y) \in Q \times Q$, consider the element $T_x Q \oplus T_yQ$
defined by
$$
s(x,y): = (- \lambda(x), \lambda(y)) \in T_x^*Q \oplus T_y^*Q
$$
that is transverse to the Liouville flow and so
induces a natural section of  the projection 
 to $Q \star Q = S^{\text{\rm can}}Q \times S^{\text{\rm can}}Q/\sim \to Q \times Q$.
\end{proof}
Therefore the $\R_+$-principal bundle $Q \star Q \to
Q \times Q$ is trivial and so isomorphic to the trivial  $\R_+$ bundle $Q \times Q \times \R_+$.

To do the calculations on the contact  product $Q \times Q \times \R$
exploiting this involutive symmetry of $\Xi$ , we also look for the associated 
contact form on the product that respects this involutive symmetry of $\Xi$.

To provide a wider perspective, we first consider the following 1-parameter family of contact forms of $\Xi$.
\begin{prop} Assume further that $(Q,\xi)$ is cooriented.
Let $\lambda$ be a contact form of $(Q,\xi)$.
For each $\kappa \in [0,1]$, we consider the map
$$
i_\kappa : Q \times Q \times \R \to  S^{\text{\rm can}}Q \times S^{\text{\rm can}}Q /\sim
$$
defined by
\be\label{eq:is}
i_\kappa(x,y,\eta) = [(-e^{(1-\kappa) \eta}\lambda(x), e^{-\kappa \eta} \lambda(y))].
\ee
Then the pull-back contact form 
$\Lambda_\kappa: =i_\kappa^*(\pi_1^*\theta + \pi_2^*\theta)$  
of the diffeomorphisms can be expressed as
$$
\Lambda_\kappa = - e^{(1-\kappa)\eta} \pi_1^*\lambda + e^{-\kappa\eta} \pi_2^*\lambda
$$
for $\kappa \in [0,1]$.
\end{prop}
\begin{proof} Recall the Liouville flow, i.e., the flow of the ODE $\dot {\bf x} = X_\theta({\bf x})$
with ${\bf x} = (x,p_x)$ on $T^*M$ is given by
$$
\psi: (t,(x,p_x)) \mapsto (x, e^t p_x)
$$
for general manifold $M$.
Consider the embedding
$$
i_\kappa: Q \times Q \times \R \to S^{\text{\rm can}}Q \times S^{\text{\rm can}}Q
$$
defined by
$$
i_\kappa (x,y,\eta):= \left(- e^{(1-\kappa) \eta} \lambda_x, e^{- \kappa \eta} \lambda_y\right).
$$
It provides a cross section for the flow $\Psi$ and so induces a diffeomorphism to $Q \star Q$,
which we still denote by the same symbol $i_\kappa$.
Furthermore we have the pull-back formula
\be\label{eq:Lambdas}
i_\kappa^*(\pi_1^*\theta + i_2^*\theta)
= - e^{(1-\kappa) \eta} \pi_1^*\lambda + e^{-\kappa \eta} \pi_2^*\lambda = \Lambda_\kappa.
\ee
Since $\ker(\pi_1^*\theta + i_2^*\theta)$ defines the contact structure on $Q \star Q$,
its pull-back under the map $i$ becomes a contact form which is precisely $\Lambda_\kappa$.
\end{proof}

\begin{rem} The case with $\kappa = 1$ corresponds to the product structure utilized in 
\cite{bhupal} (with different sign convention) and \cite{oh:entanglement1}.
In the present paper, we will focus on the case of $\kappa = \frac12$.
\end{rem}

\section{Contact product and $\Z_2$-symmetry}
\label{sec:involutive-symmetry}

To exploit the aforementioned anti-contact involution in the study of \eqref{eq:K-intro},
we will consider a suitably defined family of CR almost complex structures $\widetilde J$ 
that is adapted to a suitably chosen contact form of the contact structure $\Xi$.

For this purpose, we consider the following contact form
\be\label{eq:contact-form-tildeA}
\CA : =\Lambda_{1/2} = -e^{\frac{\eta}{2}} \pi_1^*\lambda + e^{-\frac{\eta}{2}} \pi_2^*\lambda,
\ee
instead of ${\mathscr A}$, which corresponds to \eqref{eq:Lambdas} for $\kappa = \frac12$.
 The forms satisfy the relation $\CA = e^{-\frac{\eta}{2}} {{\mathscr A}}$.
An upshot of our current choice is the following anti-symmetry carried by
$\CA$.

\begin{lem}\label{lem:anti-contact} The contact form $\CA$ satisfies
$$
\iota^*\CA = -\CA.
$$
We say that the involution $\iota$ is an \emph{$\CA$ anti-contact involution}.
\end{lem}

\subsection{$\Z_2$ anti-invariant contact form and its Reeb vector field}

The following is also straightforward to check. We would like to highlight the
sign difference in the first two components arising from the sign difference in
the two summands of $\CA = - e^{\frac{\eta}{2}}\pi_1^*\lambda + e^{-\frac{\eta}{2}}\pi_2^*\lambda$.

\begin{lem} The Reeb vector field of $\CA$ is given by
\be\label{eq:Reeb-vectorfield}
R_{\CA} = \left(- \frac12 e^{-\frac{\eta}{2}}R_\lambda,  \frac12 e^{\frac{\eta}{2}}R_\lambda, 0\right).
\ee
In particular, the Reeb flow is given by
\be\label{eq:Reeb-flow}
\phi_{\CA}^t(x,y,\eta) = \left(\phi_\lambda^{-\frac{t}{2}e^{-\frac{\eta}{2}}}(x),
\phi_\lambda^{\frac{t}{2}e^{\frac{\eta}{2}}}(y), \eta\right).
\ee
\end{lem}
\begin{proof}  From the expression $\CA = -e^{{\frac{\eta}{2}}} \pi_1^*\lambda + e^{-{\frac{\eta}{2}}} \pi_2^*\lambda$, we compute
\be\label{eq:dA}
d\CA =  -e^{{\frac{\eta}{2}}} \pi_1^*d\lambda + e^{-{\frac{\eta}{2}}} \pi_2^*d\lambda + 
\frac12 \CB \wedge d\eta
\ee
where $\CB: = e^{{\frac{\eta}{2}}} \pi_1^*\lambda + e^{-{\frac{\eta}{2}}} \pi_2^*\lambda$.
We write $R_\CA$ as 
$$
R_\CA =  \left(X^\pi + a R_\lambda, Y^\pi + b R_\lambda, c \frac{\del}{\del \eta}\right)
$$
componentwise. Substituting it into the defining equation
$$
R_\CA \rfloor d\CA = 0, \, R_\CA \rfloor \CA = 1,
$$
a straightforward calculation using \eqref{eq:dA} yields the lemma.
\end{proof}

We would like to attract readers' attention to the fact that the Reeb vector field $R_{\CA}$
is tangent to the level set of the coordinate function $\eta: M_Q \to \R$.

We also have the following special property of the structure of
chord spectrum of the pair $(M_Q,\CA)$.
\begin{lem}[Compare with Lemma 2.3 \cite{oh:entanglement1}]\label{lem:Tlambda=TA}
Let $(M_Q,\CA)$ be as above. Then we have
$$
T(Q,\lambda) = T(M_Q,\CA;\Gamma_{\text{\rm id}}).
$$
\end{lem}
\begin{proof}
Let $\gamma$ be a closed Reeb orbit of period $T$ on $Q$. Then the associated
time-dependent Hamiltonian over $[0,1]$ is $H (t,{\bf x}) \equiv  -T$.
It follows that 
$$
\widetilde \gamma(t) = \left(\gamma\left(-\frac{tT}{2}\right), \gamma\left(\frac{(1-t)T}{2}\right),0\right)
$$
is a Reeb chord of $\Gamma_{\text{\rm id}}$ and if $\gamma$ is primary, so is
$\widetilde \gamma$.

Conversely suppose $\mu:[0, T'] \to M_Q$ is a $\CA$-Reeb chord of
$\Gamma_{\text{\rm id}}$. If we write $\mu(t) = (\gamma_1(t), \gamma_2(t), \eta(t))$, then
we have
$$
\begin{cases}
\dot \gamma_1(t) = - \frac12 e^{-\frac{\eta(t)}{2}} R_\lambda(\gamma_1(t)) \\
\dot \gamma_2(t) = \frac12 e^{\frac{\eta(t)}{2}} R_\lambda(\gamma_2(t))\\
\dot \eta(t) = 0
\end{cases}
$$
and the boundary condition
$$
\gamma_1(0) = \gamma_2(0), \quad \gamma_1(T') = \gamma_2(T')
$$
for some $T' > 0$.
Therefore if $\mu(0) \in \Gamma_{\text{\rm id}}$, then $\eta(t) \equiv 0$ and
$$
\dot \gamma_1(t) = - \frac12 R_\lambda(\gamma_1(t)), \quad
\dot \gamma_2(t) = \frac12 R_\lambda(\gamma_2(t)).
$$
Let $x_0 = \gamma_1(0) = \gamma_2(0)$. Then we have
$$
\gamma_1(t) = \phi_{R_\lambda}^{-t/2}(x_0), \quad \gamma_2(t)
=  \phi_{R_\lambda}^{t/2}(x_0).
$$
Therefore the concatenated curve $\gamma:[0,T'] \to Q$ defined by
$$
\gamma(t) = \begin{cases}
\gamma_2(t) \quad & t \in [0,T'/2] \\
\gamma_1(T' - t) \quad & t \in [T'/2,T']
\end{cases}
$$
is a $\lambda$-closed Reeb orbit of period $T'$. This proves the above correspondence
is one-to-one correspondence.

Next direct calculation also shows $\int_\gamma \lambda = \CA(\widetilde \gamma)$.
This finishes the proof by the definitions of $T(M,\lambda)$
and $T_{\CA}(\widetilde \gamma)$.
\end{proof}

Lemma \ref{lem:Tlambda=TA}.
implies the following which completes a conversion of the translated point problem
into that of translated intersection problem.
\begin{cor}\label{cor:link}
$$
e^{\text{\rm trn}}_{\CA}(\Gamma_{\text{\rm id}},\Gamma_{\text{\rm id}}) \geq
T(M_Q,\CA;\Gamma_{\text{\rm id}})
\Longrightarrow e^{\text{\rm trn}}_\lambda(Q,\xi) \geq T(Q,\lambda).
$$
\end{cor}

\subsection{Hamiltonian vector field of the function $\eta$}

For this and later purpose, it is useful to introduce the notation
\be\label{eq:SA}
S_\CA := \left(\frac12 e^{-\frac{\eta}{2}} R_\lambda, \frac12 e^{\frac{\eta}{2}}R_\lambda, 0\right) \in \xi_\CA.
\ee
Then we obtain
\be\label{eq:=RS}
(0,  R_\lambda,0) =e^{- \frac{\eta}{2}} (R_\CA + S_\CA), \quad (R_\lambda,0,0) 
= e^{\frac{\eta}{2}}(S_\CA - R_\CA)
\ee
recalling $R_\CA = \left(- \frac12 e^{-{\frac{\eta}{2}}} R_\lambda, \frac12 e^{{\frac{\eta}{2}}} R_\lambda, 0\right)$.

We have the decomposition
\bea\label{eq:TQQR}
T(Q\times Q \times \R) & = & \widetilde \xi_1 \oplus \widetilde \xi_2
\oplus 
\left(\R \left\langle \frac{\del}{\del \eta}\right \rangle \oplus
\R\langle S_\CA \rangle\right)
 \oplus  \R\langle R_\CA \rangle \nonumber \\
 \Xi & = & \widetilde \xi_1 \oplus \widetilde \xi_2
\oplus 
\left(\R \left\langle \frac{\del}{\del \eta}\right \rangle \oplus
\R\langle S_\CA \rangle\right)
\eea
where $\widetilde \xi_1 = (\xi_\lambda,0,0)$ and $\widetilde \xi_2 = (0,\xi_\lambda,0)$.

\begin{lem}\label{lem:dualityinC}
The vector field $S_\CA$
is dual to $\frac{\del}{\del \eta}$ with respect to $d\CA$, i.e.,
$$
d\CA\left(S_\CA,\frac{\del}{\del \eta}\right) = 1, \, 
d\CA(S_\CA, \widetilde \xi_1 \oplus \widetilde \xi_2) \equiv 0.
$$
In particular, the decomposition of $\xi_\CA$
\be\label{eq:xiCA}
\xi_\CA = \left(\widetilde \xi_1 \oplus \widetilde \xi_2\right)
\oplus \left(\R \left\langle \frac{\del}{\del \eta}\right \rangle \oplus
\R\langle S_\CA \rangle\right)
\ee
is a symplectic orthogonal decomposition with respect to the symplectic form $d\CA$.
\end{lem}
\begin{proof} Recall the formula for $d\CA$ 
\be
d\CA =  -e^{{\frac{\eta}{2}}} \pi_1^*d\lambda + e^{-{\frac{\eta}{2}}} \pi_2^*d\lambda + 
\frac12 \CB \wedge d\eta
\ee
from \eqref{eq:dA}. We check
\be\label{eq:SCACB=1}
S_\CA \rfloor \CB = 1, \quad R_\CA \rfloor \CB = 0.
\ee
Then a direct calculation shows the lemma.
\end{proof}

\section{$\Z_2$ anti-equivariant lifting of the contact flows to the contact product}
\label{sec:lifted-Hamiltonian}
 
In this section, we investigate the relationship between the dynamics of contact Hamiltonian 
flows on $(Q,\lambda)$ and their Legendrianization on $(M_Q, \CA)$.

We have one-to-one correspondence between $\Fix^{\text{\rm trn}}(\psi)$ and the set
$\mathfrak{Reeb}(\Gamma_{\psi},\Gamma_{\text{\rm id}})$ of Reeb chords between 
$\Gamma_{\psi}$ and $\Gamma_{\text{\rm id}}$
in general \emph{with respect to the contact form $\mathscr A$.}
 (We refer readers to \cite{bhupal,sandon:translated} for the details
of this correspondence.) Since we use the different choice 
$\CA$ as the contact form on $M_Q$ to utilize a $\Z_2$-symmetry, we
will need to use some variation of this correspondence. This requires some preparation.

We recall from Observation \ref{observ:iota} that 
the contact product $M_Q = Q\star Q$ has an involutive symmetry under the
anti-contact involution 
\be\label{eq:involution-iota}
\iota(x,y,\eta) = (y,x, -\eta)
\ee
and that $\Gamma_{\text{\rm id}}$ is the fixed point set of
$\iota: M_Q \to M_Q$.

As mentioned before, to utilize this involution as in the symplectic Floer theory as in 
\cite{oh:cpam2}, \cite{fooo:anti-symplectic}, we take the contact form $\CA$ given by
$$
\CA = e^{-\frac{\eta}{2}}\mathscr A = - e^{\frac{\eta}{2}}\pi_1^*\lambda + e^{-\frac{\eta}{2}} \pi_2^*\lambda
$$
which satisfies $\iota^*\CA = - \CA$. We also recall the notation 
\be\label{eq:tildeH}
\widetilde H(t,x) = - H (1-t,x)
\ee
which generates the time-reversal flow 
\be\label{eq:time-reversal-H}
\psi_{\widetilde H}^t = \psi_H^{1-t} (\psi_H^1)^{-1}.
\ee
In this subsection, we first derive the formula of the Hamiltonian vector field $X_\H$ when $\H$ is
 the function of the type $\H = \pi^*F_t$ 
\bea\label{eq:F}
F(t,x,y) & = & - \pi^*\widetilde H(t,x) + \pi_2^*H(t,y) \nonumber \\
& = & -\widetilde H(t,x) + H(t,y) = H(1-t,x) + H(t,y)
\eea
lifted from functions $F_t$ defined on $Q \times Q$ where $\pi_1 \times \pi_2 : M_Q \to Q \times Q$,
Recall the functorial contact product $(Q \star Q,\Xi)$ is the quotient space
$$
S^{\text{\rm can}}Q \times S^{\text{\rm can}}Q/ \sim, \quad \Xi = \ker [\pi_1^*\theta \oplus \pi_2^*\theta]
$$
where $S^{\text{\rm can}}Q \subset T^*Q \setminus \{0\}$ is the canonical symplectization of $(Q,\xi)$.
With $Q$ equipped with a contact form $\lambda$, it can be identified with
$$
SQ: = Q \times \R, \quad \omega= d(e^s \pi^*\lambda), \, \pi: SQ \to Q
$$
where $(q,s) \in Q \times \R$ is the coordinates of the product.
\begin{rem}
Under this identification, it is well-known that contact Hamiltonian flow associated to
contact Hamiltonian $H$ with respect to the contact form $-\lambda$
 corresponds to the symplectic Hamiltonian flow generated by the vector field
$$
X_{\check H}^{\text{\rm symp}}(q,s) = \left(X_H^{\text{cont}}(q), R_\lambda[H](q)\right)
$$
associated to the homogeneous Hamiltonian ${\check H}(q,s)= e^s \pi^*H(q)$. 
(Appearance of the negative sign in the contact from $-\lambda$ comes from the formula \eqref{eq:dgdt}
$$
\frac{\del g_{\psi_H^t}}{\del t} = - R_\lambda[H_t] \circ \psi_H^t
$$ 
which is an artifact of our sign convention.) 
\end{rem}

We summarize the above discussion into the following.
\begin{prop}\label{prop:liftedflow} Let $F(t,x,y) = -\widetilde H(t,x) + H(t,y)$ be a time-dependent 
function on $Q \times Q$.
The contact Hamiltonian flow of $\H = \pi^*F$ on $Q \star Q$ is given by
\be\label{eq:liftedflow}
\psi_\H(t,x,y,\eta) =  [\phi_{\check{\pi^*\H}}^t(q,s)] 
\ee
where we have
$$
\phi_{\check{\pi^*\H}}^t(q,s) = \left[\left((\psi_{\widetilde H}^t(x), s -  e^{g_{\psi_{\widetilde H}^t}}(x)),
(\psi_{H}^t(y),  s - e^{g_{\psi_H^t}}(y))\right)\right].
$$
where $(t,q,s) \mapsto \phi_{\check{\pi^*\H}}^t(q,s)$ is the symplectic Hamiltonian flow on the symplectization 
$SQ = Q \times \R$ and $\pi: SQ\to Q$ is the natural projection.
\end{prop}

By pulling back the contact form 
$\Theta: = [\pi_1^*\theta \oplus \pi_2^*\theta]$ and $\H$ back to
the product $Q \times Q \times \R$ via the map $i_{\frac12}$, we obtain the following on the product
representation of the contact product $(Q \times Q \times \R, \CA)$.

\begin{cor}\label{cor:liftedflow} Let $\H$ be as above. Then we have
\bea\label{eq:liftedflow2}
\pi_1(\psi_\H(t,x,y, \eta)) & = &\psi_{e^{-{\frac{\eta}{2}}\widetilde H}}^t(x), \nonumber \\
 \pi_2(\psi_\H(t,x,y,\eta))  & = & \psi_{e^{\frac{\eta}{2}}H}^t(y), \nonumber \\
\pi_3(\psi_\H(t,x,y, \eta)) & = & \eta -g_{(\psi_{(e^{-{\frac{\eta}{2}}}\widetilde H)}^t)}(x)
 -  g_{(\psi_{(e^{{\frac{\eta}{2}}}H)}^t)}(y)
\eea
where $\pi_i$ is the natural projection to the $i$-th factor of $Q \times Q \times \R$ for $i = 1, \, 2, \, 3$.
\end{cor}

The following formula for the associated Hamiltonian vector field follows from this by
differentiating it.

\begin{prop}\label{prop:XHH} We have
$$
X_{\H}(x,y,\eta) =  \Big(X_{e^{-\frac{\eta}{2}}\widetilde F}(t,x), X_{e^{\frac{\eta}{2}} F}(t,y), 
-2 R_\CA[\H]  \Big).
$$
\end{prop}
A direct derivation of this  formula of the contact Hamiltonian vector field 
associated general Hamiltonian 
in the representation $(Q \times Q \times \R,\CA)$
is rather complicated, which is not needed for our purpose. However in Appendix \ref{sec:XH-formula},
 for readers' convenience, we include this calculation, especially for the Hamiltonians in our
 main interest of  the special type 
 $$
 \H(t,x,y,\eta) = -\widetilde H(t,x) + H(t,y)
 $$
 to illustrate the contact Hamiltonian calculus and
 the usefulness of the canonical definition of contact product $Q \star Q$.

\section{$\Z_2$ anti-invariant lifted Hamiltonian vector fields on $Q \star Q$}
\label{sec:lifted-Hamiltonian}

We have one-to-one correspondence between $\Fix^{\text{\rm trn}}(\psi)$ and the set
$\mathfrak{Reeb}(\Gamma_{\psi},\Gamma_{\text{\rm id}})$ of Reeb chords between 
$\Gamma_{\psi}$ and $\Gamma_{\text{\rm id}}$
in general \emph{with respect to the contact form $\mathscr A$.}
 (We refer readers to Section \ref{sec:contact-product} for the details
of this correspondence.) Since we use the different choice 
$\CA$ as the contact form on $M_Q$ to utilize a $\Z_2$-symmetry,
,will need to use some variation of this correspondence. This requires some preparation.

\subsection{Flattening of the Hamiltonians at $t = \frac12$}

We recall from Observation \ref{observ:iota} that 
the contact product $M_Q = Q\star Q$ has an involutive symmetry under the
anti-contact involution 
\be\label{eq:involution-iota}
\iota(x,y,\eta) = (y,x, -\eta)
\ee
and that $\Gamma_{\text{\rm id}}$ is the fixed point set of
$\iota: M_Q \to M_Q$.

As mentioned before, to utilize this involution as in the symplectic Floer theory as in 
\cite{oh:cpam2}, \cite{fooo:anti-symplectic}, we take the contact form $\CA$ given by
$$
\CA = e^{-\frac{\eta}{2}}\mathscr A = - e^{\frac{\eta}{2}}\pi_1^*\lambda + e^{-\frac{\eta}{2}} \pi_2^*\lambda
$$
which satisfies $\iota^*\CA = - \CA$. We also recall the notation 
\be\label{eq:tildeH}
\widetilde H(t,x) = - H (1-t,x)
\ee
which generates the time-reversal flow 
\be\label{eq:time-reversal-H}
\psi_{\widetilde H}^t = \psi_H^{1-t}(\psi_H^1)^{-1}.
\ee
A priori this function $\H$ is not smooth near $t = 1/2$ but we can make it so by
further flattening our isotopies at $t = 1/2$ also in addition to $t = 0, \, 1$ and then
extending it periodically as given in Hypothesis \ref{hypo:bdy-flat}. More specifically, we 
choose the time-reparameterization function $\chi: [0,1] \to [0,1]$ so that
$\chi'(t) = 0$ near $t = 0, \, \frac12, \, 1$, and then consider the Hamiltonian $H^\chi$ given by
$$
H^\chi(t,x) : = \chi'(t) H(\chi(t), x)
$$
which generates the reparameterized isotopy $\phi_H^{\chi(t)}$ and 
vanishes near $t = 0, \, \frac12, \, 1$ as hoped.

This being said, \emph{we will always assume the Hamiltonian we use $H = H(t,x)$ vanishes 
near $t = 0, \, \frac12, \, 1$ in the rest of the paper.}

\subsection{Anti-invariant lifting of Hamiltonian vector fields}

We first consider the space-time involution
$\upsilon$ on $Q \times Q \times \R \times [0,1]$ defined by
\be\label{eq:involution-upsilon}
\upsilon(x,y,\eta,t) = (y,x, -\eta, 1-t).
\ee
Such an involution is natural to consider in relation to the  Poincar\'e duality which 
was used in the \emph{chain-level} symplectic Hamiltonian Floer theory 
by the author in \cite{oh:cag}. This is a contact counterpart thereof.

Let $H = H(t,x)$ be a 1-periodic Hamiltonian on $(Q,\lambda)$ and 
consider the Hamiltonian 
$$
\H(t,x,y): = - \widetilde H(t,x) + H(t,y) 
= H(1-t,x) + H(t,y)
$$
on the contact product $M_\Q = Q \times Q \times \R$.
One of the upshots of this choice of 1-Hamiltonian is  the following.

\begin{prop}\label{prop:symmetry} $\H$ satisfies $\H \circ \upsilon = \H$ and hence,
$$
\H_{1-t}\circ \iota =  \H_t, \quad \iota^*X_{\H_{1-t}} = - X_{\H_t}, \quad
\, \iota^{-1} \circ \psi_{\widetilde \H}^t \circ \iota = \psi_\H^t
$$
\end{prop}
\begin{proof}
We just evaluate
\beastar
\H \circ \upsilon(t,x,y,\eta) & = & \H (1-t,y,x,-\eta) = - \widetilde H(1-t,y) +H(1-t,x)\\
& = & H(t,y) - \widetilde H(t,x) = \H(t,x,y,\eta)
\eeastar
by definition. This finishes the first statement.

To prove the second statement, 
by the defining property of Hamiltonian vector field, it is enough to prove
\be\label{eq:iotaXH1-t}
\begin{cases}
\iota^*X_{\H_{1-t}} \rfloor d\CA = - X_{\H_t} \rfloor d\CA\\
\iota^*X_{\H_{1-t}} \rfloor \CA =  \H_t.
\end{cases}
\ee
For the second identity of this, we just evaluate
$$
\iota^*X_{\H_{1-t}} \rfloor \CA = -\iota^*(\CA(X_{\H_{1-t}}) ) = - \iota^*(-\H_{1-t}) = \H_t 
$$
To prove the first identity, using the property that $\iota$ preserves the contact distribution, we compute
\beastar
(\iota^*X_{\H_{1-t}} \rfloor d\CA) (Y(x)) & = & d\CA(d\iota (X_{\H_{1-t}})(\iota(x)),Y^\pi(x)) \\
& = &
\iota_*d\CA(X_{\H_{1-t}}(x), \iota_*(Y^\pi))(x)) \\
& = & - d\CA (X_{H_{1-t}}(x), (\iota_*Y)^\pi(x)) \\
& = & - (X_{\H_{1-t}}(x) \rfloor d\CA)( (d\iota(Y(\iota(x))^\pi) \\
& = & - (d\H_{1-t} - R_\CA[\H_{1-t}] \CA) ((d\iota(Y(\iota(x))^\pi) \\
& = & - d(\H_{1-t}\circ \iota) (Y(x)^\pi) = -d\H_t(Y(x)^\pi) \\
& = & -(d\H_t - R_\CA[\H_t]\, \CA)(Y^\pi)(x) \\
& = &  - (X_{\H_t} \rfloor d\CA)(Y^\pi (x)) = - (X_{\H_t} \rfloor d\CA)(Y(x)).
\eeastar
This proves
$$
(\iota^*X_{\H_{1-t}} \rfloor d\CA) = - X_\H \rfloor d\CA.
$$
Combining the two, we have proved \eqref{eq:iotaXH1-t} and hence 
$\iota^*X_{\H_{1-t}} =  - X_\H$. The equality $ \iota^{-1} \circ \psi_{\widetilde \H}^t \circ \iota = \psi_\H^t$
is an immediate consequence of this by integration.
\end{proof}

\begin{rem}\label{rem:flatat1/2}
Such a space-time symmetry, especially the tilde operation \eqref{eq:tildeH},
plays an important role in the study of \emph{chain level} Poincar\'e
duality isomorphism in the Floer homology theory in symplectic geometry
(see \cite{oh:cag}). It also plays a crucial role in proving the symmetry property
$$
\gamma(\phi)= \gamma(\phi^{-1})
$$
of the spectral norm $\gamma$ defined in \cite{oh:dmj}. We need to consider $\widetilde H$
by the similar reason in the present contact context.
\end{rem}

\subsection{Correspondence between ${\Fix}^{\text{\rm tr}}(\psi)$ and $\mathfrak{Reeb}(\Gamma_\psi,\Gamma_\id)$}

Now we examine the boundary condition. The following is another upshot of
the choice of the Hamiltonian $\H$. Let $\psi = \psi_H^1$. We define
$$
R_0 = \Gamma_{\psi}: = \{(\psi(q), q, g_{\psi}(q))\mid q \in Q\}, \quad R_1 = \Gamma_\id.
$$
\begin{prop} \label{prop:time1/2-map}  
$$
\psi_{(\H;\CA)}^{\frac12}(\Gamma_\psi) = \Gamma_\id.
$$
\end{prop}
\begin{proof} This is an immediate consequence of the formula \eqref{eq:liftedflow2} given in Corollary \ref{cor:liftedflow}
by applying $(\psi_{(\H;\CA)}^{\frac12})^{-1}$ to $(x,y,\eta) = (q,q,0)$ thereto from the expression
$$
(\psi_{\widetilde H}^{\frac12})^{-1}(q) = \psi_H^1 (\psi_H^{(1-\frac12)})^{-1} 
$$
and
$$
g_{(\psi^{\frac12}_{\widetilde H})^{-1}}(q) = g_{\psi_H^1 (\psi_H^{\frac12})^{-1})}(q) = g_{\psi_H^1}(\psi_H^{\frac12})^{-1}(q)) 
+ g_(\psi_H^{\frac12})^{-1}(q)):
$$
First, using the formula \eqref{eq:liftedflow2}, we compute
$$
(\psi_{(\H;\CA)}^{\frac12})^{-1}(q,q,0) = 
\left(\psi_H^1 ((\psi_H^{(1-\frac12)})^{-1}(q)), (\psi_H^{(1-\frac12)})^{-1}(q),  g_{\psi_H^1}(\psi_H^{\frac12})^{-1}(q)\right).
$$
Then, by putting $q' = (\psi_H^{(1-\frac12)})^{-1}(q)$, we have shown
$$
(\psi_{(\H;\CA)}^{\frac12})^{-1}(q,q,0) = \left(\psi_H^1(q'),q',g_{\psi_H^1}(q')\right).
$$
This proves $\left(\psi_{(\H;\CA)}^{\frac12}\right)^{-1}(\Gamma_\id) = \Gamma_\psi$
which is equivalent to the required statement.
\end{proof}

Now we examine the boundary condition. The following is another upshot of
the choice of the Hamiltonian $\H$.

We  consider the pair $(R_0,R_1)$ of Legendrian submanifolds 
The following proposition shows the relationship between translated points of $\psi$ and
the Reeb chords for the pair $(R_0,R_1)$.

\begin{prop}\label{prop:Fix-Reeb-correspondence} Let $\psi = \psi_H^1$.
We have one-correspondence
$$
\Fix^{\text{\rm trn}}(\psi) \longleftrightarrow \mathfrak{Reeb}(\Gamma_\psi,\Gamma_\id) \longleftrightarrow
\mathfrak X^{\text{\rm trn}}((\Gamma_\id, \Gamma_\id);H).
$$
\end{prop}
\begin{proof}
Let $(\gamma,T)$ be a Reeb chord from $R$ to $R'$. By definition, this means that there
is a pair of points $q, \, q' \in Q$ such that
$$
\phi_{R_\CA}^T \left(\psi(q),q, g_{\psi}(q)\right) = (q',q', 0)
$$
This equation can be rewritten as
\be\label{eq:translated-point-eq}
\phi_{R_\lambda}^{-T/2}(\psi(q)) =
\phi_{R_\lambda}^{T/2}(q), \quad g_{{\psi}}(q) = 0
\ee
after eliminating $q'$. Therefore we have derived
$$
\phi_{R_\lambda}^T(q) = \psi(q), \quad g_{{\phi_H^1}}(q) = 0
$$
and hence $q$ is a translated fixed point. 
For the opposite direction, we read the above proof backwards. 
This provides the first one-one correspondence.

For the second correspondence, we recall the definition of $\mathfrak{X}(\Gamma_\id,\Gamma_\id;\CA)$
from Definition \ref{defn:trans-Ham-chords}. Then we have only to rewrite the last 
equation as
$$
\psi(\phi_{R_\lambda}^T(q)) = q, \quad g_{{\phi_H^1\phi_{R_\lambda}^T}}(q) = 0.
$$
This finishes the proof.
\end{proof}

We have one-to-one correspondence between $\Fix^{\text{\rm trn}}(\psi)$ and the set
$\mathfrak{Reeb}(\Gamma_{\psi},\Gamma_{\text{\rm id}})$ of Reeb chords 
between $\Gamma_{\psi}$ and $\Gamma_{\text{\rm id}}$
in general. (We refer readers to Section \ref{sec:involutive-symmetry} for the details
of this correspondence.)

We recall from Observation \ref{observ:iota} that 
the contact product $M_Q = Q\star Q$ has an involutive symmetry under the
anti-contact involution 
\be\label{eq:involution-iota}
\iota(x,y,\eta) = (y,x, -\eta)
\ee
and that $\Gamma_{\text{\rm id}}$ is the fixed point set of
$\iota: M_Q \to M_Q$.

With Proposition \ref{prop:time1/2-map} in our disposal, which relates $\Gamma_\psi$ and $\Gamma_\id$
at time $t = \frac12$, we rescale the Hamiltonian in time and consider the Hamiltonian
\be\label{eq:HH'}
\H'(t,x,y) := \frac12 \H\left(\frac12t,x, y\right) = - \frac12 \widetilde H \left(1 - \frac{t}{2},x\right)
 + \frac12 H\left(\frac{t}{2},y\right).
\ee
Then Proposition \ref{prop:liftedflow}  is equivalent to the equality
\be\label{eq:time1-map}
\psi_{\H'}^1(\Gamma_\psi) = \Gamma_\id
\ee
holds. From now on, we will take this $\H'$ as the lifted Hamiltonian associated to
the given Hamiltonian $H = H(t,q)$ on $Q$.

\part{Anti-contact involution and proof of Shelukhin's conjecture}
\label{part:shelukhin-conjecture}

In this part, we first recall the basic construction of parameterized moduli
spaces used in our proof of Theorem \ref{thm:shelukhin-intro} in \cite{oh:entanglement1}, which will be
also used in our proof of Theorem \ref{thm:dual-pair}. In Appendix \ref{sec:recollection},
for readers' convenience,
we duplicate the description of the moduli space and the methodology of its use
in the study of Sandon-Shelukhin type problems employed in \cite{oh:entanglement1}.

After that, we will specialize the Legendrian submanifolds to
the case of Legendrian graph $\Gamma_\psi$ of contactomorphism 
$\psi$ contact isotopic to the identity
and explain how we adapt the aforementioned construction to the present case
in the $\Z_2$-equivariant way utilizing the natural \emph{anti-contact involution}.
We then exploit this $\Z_2$-symmetry to overcome non-tameness of  $M_Q$.

We put the following hypothesis without loss of generalities
by making an arbitrarily $C^\infty$-small perturbation of the given Hamiltonian.

\begin{hypo}\label{hypo:no-fixedpoint} Assume
$$
\Fix \psi_H^1 = \emptyset.
$$
\end{hypo}
This hypothesis is equivalent to the non-intersection
\be\label{eq:non-intersection}
\Gamma_\id \cap \Gamma_\psi = \emptyset, \quad \psi = \psi_H^1
\ee
which is also imposed in \cite[Hypothesis 8.6]{oh:entanglement1} in its proof of 
Theorem \ref{thm:shelukhin-intro}.

\section{Tameness of lifted CR almost complex structures}
\label{sec:tameness}

We first duplicate the general
definitions  of tame contact manifolds given in the introduction which itself is
borrowed from \cite{oh:entanglement1}.
 
\begin{defn}[Reeb-tame functions] Let $(M,\lambda)$ be a contact manifold.
A function $\varphi: M \to \R$ is called \emph{$\lambda$-tame} on an open subset $U \subset M$ if $\CL_{R_\lambda} d\varphi = 0$ on $U$.
\end{defn}

\begin{defn}[Contact {$J$} quasi-pseudoconvexity]
Let $J$ be a $\lambda$-adapted CR almost complex structure. We call a function $\varphi: M \to \R$
\emph{contact $J$ quasi-plurisubharmonic on $U$}
\bea
-d(d \varphi \circ J) + k d\varphi \wedge \lambda & \geq & 0 \quad \text{\rm on $\xi$}, \label{eq:quasi-pseuodconvex-intro}\\
R_\lambda \rfloor d(d\varphi \circ J) & = & g \, d\varphi \label{eq:Reeb-flat-intro}
\eea
for some functions $k, \, g$ on $U$. We call such a pair $(\varphi,J)$
a \emph{contact quasi-pseudoconvex pair} on $U$.
\end{defn}
These properties of $J$ enable us to apply the maximum principle
(see \cite{gilbarg-trudinger}) in the study of the analysis of the moduli space
of contact instantons associated to $J$.  We refer readers to \cite{oh:entanglement1} for more detailed discussion on the
analysis of contact instantons on tame contact manifolds.

\begin{defn}[Tame contact manifolds, Definition 1.12 \cite{oh:entanglement1}]\label{defn:tame}
Let $(M,\xi)$ be a contact manifold, and let $\lambda$ be a contact form of $\xi$.
\begin{enumerate}
\item  We say $\lambda$ is \emph{tame on $U$}
if $(M,\lambda)$ admits a \emph{contact $J$-quasi pseudoconvex pair} on $U$ with a \emph{proper}
and $\lambda$-tame $\varphi$.
\item We call an end of $(M,\lambda)$ \emph{tame} if $\lambda$ is tame on the end.
\end{enumerate}
We say an end of contact manifold $(M,\xi)$ is tame if it admits a contact form
$\lambda$ that is tame on the end of $M$.
\end{defn}

Then we will apply the above definitions and the maximum principle
 to $M \times M \times \R$. Taking the family
$$
\H^{'s}(t,{\bf x})= \Dev_\CA \left(t \mapsto \psi_{\H'}^{st}({\bf x})\right)
$$
and
$$
\H^{\prime\chi}_K(\tau,t,{\bf x}) = \Dev_\CA \left(t \mapsto \psi_{\H'}^{\chi_K(\tau)t}({\bf x})\right)
$$
with the standard elongation function $\chi:\R \to [0,1]$ used as in \cite{oh:entanglement1},
we denote the $\tau$-Hamiltonian by
$$
{\mathbb G}_K(\tau,t,{\bf x}) = \Dev_{\CA}(\tau 
\mapsto \psi_{\H'}^{\chi_K(\tau)t})
$$
as before.
For the choice of adapted CR almost complex structures for the equation \eqref{eq:K}
to properly exploit the anti-contact involution $\iota$, we take the contact form $\CA$
of $\Xi = \ker {\mathscr A}$.

We next recall  the gauge transformation of a map $u: \R\times [0,1] \to M$
for a given time-dependent Hamiltonian $H = H(t,x)$ on $M$,
\be\label{eq:ubarK}
\overline u_K(\tau,t): = \left (\psi_H^{\chi_K(\tau)t}\right)^{-1}(u(\tau,t))
\ee
and make the choice of the family contact triads  $(M,\lambda', J')$
as follows.
 
\begin{choice}[Choice \ref{choice:lambda-J}]\label{choice:lambdaJ'} For given $J = \{J_t\}$
we consider the following two parameter families of $J'$ and $\lambda'$:
\bea
J_{(s,t)}' & =  & (\psi_{H^s}^t)_*J \label{eq:J-prime}\\
\lambda_{(s,t)}' & = & (\psi_{H^s}^t)_*\lambda. \label{eq:lambda-prime}.
\eea
\end{choice}
The upshot of this gauge transformation is that it makes one-to-one correspondence between
the set of solutions of \eqref{eq:perturbed-contacton-bdy} and that of \eqref{eq:contacton-bdy-J0}, and for their
2-parameter version thereof. (See \cite[Section 5]{oh:entanglement1} for the detailed
comparison between the two.)

Because of this, we will often denote by
$\CM^{\text{\rm para}}_{[0,K_0+1]}(M,\lambda;R,H)$
both the set of solutions of \eqref{eq:K} and that of its gauge transform \eqref{eq:contacton-bdy-J0}
freely without distinction throughout the paper. However for the clarity, we will always denote by $U$
a solution of the former equation and by $\overline U$ that of the latter later when
we work with the product $M_Q = Q \times Q \times \R$. We would like to mention that
the $\Z_2$ anti-symmetry are easier to see and more apparent for the dynamical version \eqref{eq:K} of 
perturbed CI equation.  We will work with the dynamical version of the CI equation
\be\label{eq:K}
\begin{cases}
(dU - X_{{\H'}_K}(U)\, dt - X_{\mathbb G_K}(U)\, ds)^{\pi(0,1)} = 0, \\
d\left(e^{\widetilde g_K(U)}(U^*\CA + U^*{{\H'}}_K dt
+ U^*{\mathbb G}_K\, d\tau) \circ j\right) = 0,\\
U(\tau,0), \quad U(\tau,1)  \in R_0.
\end{cases}
\ee
to exploit the symmetry, because the gauge
transform destroys the $\Z_2$ symmetry. On the other hand,
we work with the geometric version \eqref{eq:contacton-bdy-J0} of the (perturbed) CI equation
for the application of $C^0$-estimates because it is
easier to work with for the latter purpose.

We now consider the contact product $M_Q$ and the following class of $\CA$-adapted almost complex
structures. A direct calculation shows the description of contact structure $\Xi = \ker \CA$
\bea\label{eq:tildexi}
\ker \CA = \{(X,Y,0) \mid X, \, Y \in \ker \lambda \} \oplus
\, {\text{\rm span}}_\R \left\{S_\CA,\frac{\del}{\del \eta}\right\} =: \Xi_1 \oplus \Xi_2
\eea

\begin{defn}[Lifted CR almost complex structure]\label{defn:tildeJ}
Let $J$ be any $\lambda$-adapted CR almost complex structure. Consider the
$\CA$-adapted CR almost complex structure $\widetilde J$ characterized by
\be\label{eq:tildeJ-XY}
\widetilde J(X,0,0) = (JX, 0, 0), \quad \widetilde J(0,Y,0) = (0, JY ,0),
\ee
for $X, \, Y \in \xi_\lambda$ and
\be\label{eq:tildeJSA}
\widetilde J S_{\CA} =  - \frac{\del}{\del \eta}, \quad 
\widetilde J \left(\frac{\del}{\del \eta} \right) = S_{\CA}
\ee
and
\be\label{eq:tildeJ-rest}
0 = \widetilde J R_{\CA}=
\widetilde J\left(-\frac12 e^{-\frac{\eta}{2}} R_\lambda, \frac12 e^{\frac{\eta}{2}}R_\lambda, 0\right)
\ee
where we recall the definition of the vector field
$$
S_{\CA}(x,y,\eta) = \left (\frac12 e^{-\frac{\eta}{2}} R_\lambda(x), \frac12 e^{\frac{\eta}{2}}R_\lambda(y), 0\right).
$$
We call  $\widetilde J$ a lifted CR almost complex structure of $J$ from $Q$.
\end{defn}

\begin{lem}\label{lem:tildeJ}  $\widetilde J$ is an $\CA$-adapted CR-almost complex structure.
Furthermore we have
\be\label{eq:tildeJ-symmetry}
 \iota^*\widetilde J = \begin{cases} \widetilde J \quad & \text{on } \Xi_1 \\
-\widetilde J \quad & \text{on } \Xi_2.
\end{cases}
\ee
\end{lem}
\begin{proof} 
It is straightforward to check $\widetilde J^2 = -\id$ from definition and so its proof is omitted.
It remains to check the identity \eqref{eq:tildeJ-symmetry}.

We can check this by evaluating $\iota^*\widetilde J(X,Y,a \frac{\del}{\del \eta})$
separately on $\Xi_1$ and $\Xi_2$
noting the $\widetilde J$ preserves the decomposition by definition of $\widetilde J$.

It is straightforward to check $\iota^*{\widetilde J} = \widetilde J$ on $\Xi_1$ because 
$\iota(x,y,\eta) = (y,x, -\eta)$ which shows that $d\iota(v_x,v_y, 0) = (v_y, v_x, 0)$ for 
$v_x \in \xi_x, \, v_y \in \xi_y$.

We will just focus on the proof of the identity
\be\label{eq:tildeJSA-involution}
\iota^*\widetilde J(S_\CA) = -\widetilde J(S_\CA)
\ee
using the definition \eqref{eq:tildeJSA}.
Now we compute
\beastar
\iota^*\widetilde J(S_\CA|_{(x,y,\eta)}) & = & d\iota \widetilde J d\iota(S_\CA|_{(x,y,\eta)})\\
& = & d\iota \widetilde J d\iota\left(\frac12 e^{-\frac{\eta}{2}}R_\lambda(x), \frac12 e^{\frac{\eta}{2}}R_\lambda(y), 0\right)\Big|_{(x,y,\eta)} \\
& -= & d\iota \widetilde J \left(\frac12 e^{\frac{\eta}{2}}R_\lambda(y), \frac12 e^{-\frac{\eta}{2}}R_\lambda(x), 0\right)\Big|_{(y,x,-\eta)}\\
& = & d\iota \widetilde J(S_\CA|_{(y,x,-\eta)}) = 
 d\iota \left(-\frac{\del}{\del \eta}\Big|_{(y,x,-\eta)}\right) \\
& = &\frac{\del}{\del \eta}\Big|_{(x,y,\eta)} = - \widetilde J(S_\CA|_{(x,y,\eta)}).
\eeastar
This finishes the proof of \eqref{eq:tildeJSA-involution}. 
\end{proof}

The next proposition shows that $(M_Q,\CA)$ is tame.
\begin{prop}[Proposition 11.7 \cite{oh:entanglement1}]
\label{prop:tame}
The contact manifold $(M_Q,\CA)$ is tame in the sense of
Definition \ref{defn:tame}.
\end{prop}
\begin{proof}  Clearly $\CL_{R_\lambda} d\eta = d(R_\lambda \rfloor d\eta) = 0$.
A straightforward computation using the definitions \eqref{eq:tildeJ-XY} - \eqref{eq:tildeJ-rest} 
and \eqref{eq:SCACB=1} shows
$$
-d\eta \circ \widetilde J = \CB.
$$
By taking the differential of this, we obtain
\bea\label{eq:ddetatildeJ}
-d(d\eta \circ \widetilde J) & = & d\CB = d\left(e^{\frac{\eta}{2}}\pi_1^*\lambda + e^{-\frac{\eta}{2}}\pi_2^*\lambda\right)
\nonumber\\
& = &  e^{\frac{\eta}{2}}\pi_1^*d\lambda + e^{-\frac{\eta}{2}}\pi_2^*d\lambda + \frac12 d\eta \wedge(e^{\frac{\eta}{2}}\pi_1^*d\lambda - e^{-\frac{\eta}{2}}\pi_2^*d\lambda ) \nonumber \\
& = & e^{\frac{\eta}{2}}\pi_1^*d\lambda + e^{-\frac{\eta}{2}}\pi_2^*d\lambda - \frac12 d\eta \wedge \CA.
\eea
This shows
\be\label{eq:plurisubharmonicity}
-d(d\eta \circ \widetilde J) + \frac12 d\eta \wedge \CA
= e^{\frac{\eta}{2}}\pi_1^*d\lambda + e^{-\frac{\eta}{2}}\pi_2^*d\lambda \geq 0.
\ee
In particular, the function  $\varphi(x,y, \eta)) = |\eta|$ is contact $J$-quasiconvex: When $\eta > 0$,
this is apparent from \eqref{eq:plurisubharmonicity}. When $\eta < 0$, we have
$\varphi(x,y,\eta) = -\eta$ which is a positive proper function on the region $\eta < 0$.
Furthermore, substituting replacing $\eta$ by $-\eta$, we have derived
$$
-d(d\eta \circ \widetilde J) + \frac12 d\eta \wedge \CA
= e^{\frac{- \eta}{2}}\pi_1^*d\lambda + e^{\frac{\eta}{2}}\pi_2^*d\lambda
$$
which is positive. Combining the two, we have derived
$$
d(\varphi \circ \widetilde J) + \frac12 d\varphi \wedge \CA
= e^{\frac{\varphi}{2}}\pi_1^*d\lambda + e^{-\frac{\varphi}{2}}\pi_2^*d\lambda
$$
on $\{\eta \neq 0\}$. Finally we compute
$$
R_{\CA} \rfloor d(d\varphi \circ \widetilde J)
= R_\CA\rfloor \left(  e^{\frac{\varphi}{2}}\pi_1^*d\lambda + e^{-\frac{\varphi}{2}}\pi_2^*d\lambda 
- \frac12 d\varphi \wedge \CA \right) \\
=  \frac12 d\varphi.
$$
Therefore the function $\varphi = |\eta|$ is tame and $J$ quasi-plurisubharmonic on the region $\{\eta> 0\}$.
Furthermore it is proper $\{|\eta| \neq 0\}$ since we assume $Q$ is compact. 
(Recall the zero set $\{\eta = 0\} = \Gamma_\id$ is compact.)
This implies that each subset $\{0 \leq |\eta| \leq c\}$ for any $c \geq 0$ is compact. 
 This finishes the proof.
\end{proof}

\begin{rem} It turns out that the choice of the correct class of lifted almost
complex structures that make the contact product requires something nontrivial
\emph{to make the relevant triad $(M,\lambda,J)$ satisfy the contact $J$ quasi-pseudoconvexity}.
The choice $J$ made by the author asking 
$$
\widetilde J(X,0,0) = (-JX,0,0)
$$
in  the previous arXiv postings is not a right choice in that the right hand side \eqref{eq:plurisubharmonicity}
fails to be positive because of the negative sign above. We thank the unknown referee for pointing this out.
\end{rem}

\section{Lifted CR almost complex structures and $C^0$ estimates}
\label{sec:C0-estimates}

By now, we have achieved our goal of defining 
the natural class of CR almost complex structures satisfying \eqref{eq:tildeJ-symmetry}
by the definition of the lifted $\widetilde J$ and the choice of the Hamiltonian $\H'$
which satisfies $\upsilon^*\H' = \H'$.

We then take the gauge transformations of the triads $(M_Q, \widetilde J, \CA)$
as follows.

\begin{choice}[Choice of lifted CR almost complex structures ${\widetilde J}'$] 
We take the \emph{domain-dependent} contact triad
 $$
 (Q,\{J_z\},\{\lambda_z\})
 $$
 such that $(J'_z,\lambda'_z) \equiv (J_0,\lambda)$ outside a compact subset near
 punctures of $\dot \Sigma$ in Choice \ref{choice:lambdaJ'}.  
 Then we lift them to a domain-independent one $(M_Q, \widetilde J', \CA)$ 
 for the equation \eqref{eq:K} and  their gauge transformation
 $(\{{\widetilde J}_z\}, \{{\CA}_z\})$ so that the lift satisfies:
 \begin{enumerate}
 \item The formula \eqref{eq:J-prime} on a  compact neighborhood of
 $$
 \bigcup_{(s,t) \in [0,1]^2} \Gamma_{\psi_H^{st}} \subset Q \times Q \times \R,
 $$
 \item $(\{{\widetilde J}_z\}, \{{\CA}_z\}) \equiv (\widetilde J', \CA)$ outside a possibly bigger
 compact subset thereof and so defines a (domain-dependent) tame pairs.
 \end{enumerate}
\end{choice}

\begin{rem}
This choice of time-independent $J' = \widetilde J_0$ outside a compact subset will be needed
to apply the maximum
principle to achieve the $C^0$-estimate for the equation
\eqref{eq:K}
for the noncompact contact manifold $M_Q = Q \times Q \times \R$. In fact, since $\H$ is
assumed to be compactly supported, it also holds that 
the gauge transformation is trivial and so $\widetilde J_z \equiv \widetilde J_0$
for all $z \in \dot \Sigma$ outside a compact subset of $M_\Q$.
\end{rem}

Let $\widetilde {\CA} = \{\CA_z\}$ and $\H'$ be defined as above. 
Then we consider the equation \eqref{eq:K} associated to $(M_Q,\Gamma_{\text{\rm id}})$,
 $\widetilde J$ and ${\H'}^s$  given in \eqref{eq:HH'}.

In the reset of the present section, we establish an explicit uniform $C^0$ bound of solutions $U$.

To concisely state the $C^0$-estimate we prove, we recall that noncompactness of $M_Q$
arises from the $\eta$-direction and that $\Gamma_\id \subset \{\eta = 0\}$.
 We introduce the following constant
\beastar
C(\{g_{H^s}\}) &: = & \sup_{(s,t,x)} |g_{(H^s;\lambda)}(t,x)| = \max_{s \in [0,1]} \| g_{(H^s;\lambda)}\|_{C^0},\\
C(\{(H^s,\widetilde J^s\}) &: = & \sup \{|\eta| \mid (H^s(x,y,\eta), \widetilde J^s(x,y,\eta)) \neq (0, J_0) \}
\eeastar
and then define
\be\label{eq:constant-C}
C(\{(H^s,\widetilde J^s,g_{H^s})\}): = \max\left\{C(\{g_{H^s}\}), C(\{(H^s,\widetilde J^s\})\right\}.
\ee
\begin{prop} Let  $W$ be a solution to the geometric CI equation, the one \eqref{eq:contacton-bdy-J0} for $\overline U$, associated to
the dynamical one \eqref{eq:K}. Then we have
$$
\|\eta\circ W\|_{C^0} \leq C(\{H^s,\widetilde J^s,g_{H^s}\}).
$$
In particular, $\Image W$ is precompact in $M_Q = Q \times Q \times \R$.
\end{prop}
\begin{proof} 
We have seen that $\eta$ is tame and $\widetilde J$-plurisubharmonic
on the region $\eta > 0$ by Proposition \ref{prop:tame}. Then it follows from \cite[Theorem 6.1]{oh:entanglement1} that 
$\eta \circ W$ is a  subharmonic function
for any contact triad $W$.
We define the positive and the negative parts of $W$ by
$$
(\eta\circ W)_+(z) := \max\{\eta \circ W(z), \, 0\}, \quad (\eta \circ W)_-: = \min\{\eta \circ W(z),  0\}
$$
respectively.
By the maximum principle applied to the function $\eta \circ W$ for $\eta > 0$, we
obtain
$$
\sup_{z \in \dot \Sigma} (\eta \circ W)_+(z) \leq
\max \left\{\sup_{z \in \del \dot \Sigma} (\eta \circ W(z))_+, \max \{\eta(\gamma_\pm\}, \, 0\right\}.
$$
Now we examine the values of $(\eta\circ W)_+$ on the boundary and at infinity 
of $\dot \Sigma = \R \times [0,1]$.

On the boundary component $\del \dot \Sigma = \dot \Sigma \cap \{t=0, \, 1\}$, the image of $W$
is contained in $\cup_{s \in [0,1]} R_0^s$ and $R_1$ respectively, we have
$$
\sup_{z \in \del \dot \Sigma} (\eta \circ W)_+(z) \leq C(\{R_0^s\})
$$
Furthermore since $R_{\CA}$ is tangent to the level sets of $\eta$, we also have
$$
\max_{t \in [0,1]} (\eta\circ W)_+(\pm \infty, t) 
= \max\{(\eta\circ W(\pm \infty,0))_+, (\eta\circ W(\pm\infty),1))_+\}
\leq C(\{H^s\})
$$
where we use the observation that $t \mapsto W(\pm \infty, t)$ are Reeb chords from $\Gamma_{\psi}$ to
$\Gamma_{\text{\rm id}}$ and hence we can write
$$
W(\pm \infty,t) = \phi_{R_\CA}^t \left(x_\pm, x_\pm, 0 \right)
$$
with $g_\psi(x_{\pm x}) = 0$. 
for some  $\eta_\pm > 0$ and $x_\pm \in Q$ with $\phi_{R_\lambda}^{\eta_\pm}(x_\pm) = \psi(x_\pm)$ where $\psi = \psi_H^1$.
(Recall the standing hypothesis, \eqref{eq:non-intersection}, which implies
that the associated iso-speed Reeb chord cannot be a constant path.)

Combining the above, we have proved
$$
\sup_{z \in\dot \Sigma} (\eta \circ W)_+(z) \leq C(\{H^s\})
$$
which proves the boundedness from above.

It remains to show the boundedness from below, i.e.,
\be\label{eq:inf-etaW}
\inf_{z \in \dot \Sigma} (\eta \circ W)_-(z)\geq -C(\{H^s\})
\ee
or equivalently
$$
\sup_{z \in \dot \Sigma}(- \eta \circ W)_+(z) \leq C(\{H^s\}).
$$
This again follows from the maximum principle applied to $-\eta$ for $\eta < 0$.
This proves the uniform lower bound and finishes the proof.
\end{proof}

Obviously this applies to the gauge transformation $\widetilde U$ of $U$
which is also a solution to \eqref{eq:K}.

\begin{cor}\label{eq:C0-estimate-K}
Let  $U$ be a solution to \eqref{eq:K}. Then we have
$$
\|\eta\circ U\|_{C^0} \leq C(\{(H^s,\widetilde J^s,g_{H^s}\}).
$$
In particular, $\Image U$ is precompact in $M_Q = Q \times Q \times \R$.
\end{cor}
Once this $C^0$-bound in our disposal, we can apply the whole package established in
\cite{oh:entanglement1} the essentials of which we duplicate in Section \ref{sec:recollection}.

We now apply the relevant deformation-cobordism analysis in the package
to the case of $M_Q = Q \times Q \times \R$
 whose details are now in order.

\section{Energy estimates and wrap-up of the proof}
\label{sec:energy-estimates}

Next we study the $C^1$-estimates and a weak convergence result for
the parameterized moduli space
$$
\CM^{\text{\rm para}}(M,R;J,H) = 
\CM_{[0,K_0]}^{\text{\rm para}}(M,R;J,H) = \bigcup_{K \in [0, \infty)} \{K\} \times
\CM_K(M,R;J,H)
$$
defined in \eqref{eq:M-para-K0} which corresponds to the case of $K_0 = \infty$.

For this purpose, we first need to establish the energy estimates.
As in \cite{oh:entanglement1}, we divide the estimates into the two
components, the $\pi$-energy and the vertical energy components.

We start with the following general energy identity derived
in \cite{oh:entanglement1} which is applied to $\H'$ in particular. 

\begin{lem}[Proposition 8.6 \cite{oh:entanglement1}]
\label{lem:pienergy-bound2}
Let $U$ be any finite energy solution of \eqref{eq:K}. Then we have
\bea\label{eq:EJKU=}
E_{(\widetilde J_K,\H'_K)}^\pi(U) & = &
\int_{-2K-1}^{-2K} - \chi_K'(\tau) \left( \int_0^1
\H'\left(r,\psi_{\H'}^{r \chi_K(\tau)}(U(\tau,0))\right)\, dr\right)\, d\tau \nonumber\\
&{}& \quad + \int_{2K}^{2K+1} - \chi_K'(\tau) \left( \int_0^1 \H'\left(r,\psi_ {\H'}^{r \chi_K(\tau)}(U(\tau,0))\right)\, dr\right)\, d\tau.
\nonumber\\
{}
\eea
\end{lem}

Using this general identity, we prove the following energy estimate.

\begin{prop}\label{prop:pi-energy-bound}
Let $U$ be any finite energy solution of \eqref{eq:K}. Then we have
\be\label{eq:pi-energy-bound}
E^\pi(\overline U) \leq \|H\|
\ee
for all $K \geq 0$.
\end{prop}
\begin{proof}
We recall that $U(\tau,0), \, U(\tau,1) \in R_0$. We write $U = (u_1, u_2, \eta)$ and then
$$
U(\tau,0) = (u_1(\tau,0), u_2(\tau,0), 0), \quad U(\tau,1) = (u_1(\tau,1), u_2(\tau,1), 0).
$$
Here we would like to recall
\be\label{eq:psitildeH}
\psi_{\widetilde H}^{\chi_K(\tau) r} (u_1(\tau,0))
=\psi_{H}^{\chi_K(\tau)(1-r)}(\psi_{H}^{\chi_K(\tau)})^{-1}.
\ee
Substituting these into \eqref{eq:EJKU=} and applying the formula of $\H'$ 
 we derive
\beastar
E_{(\widetilde J_K,{\H'}_K)}^\pi(u) & = & \int_{-2K-1}^{-2K} -\chi_K'(\tau)
\left(\int_0^1 \frac12 H\left(1-\frac{r}{2}, \pi_1\psi_{\H'}^{r \chi_K(\tau)}
(u_1(\tau,0))\right)
\, dr \right)\, d\tau \\
&{}& + \int_{2K}^{2K+1} - \chi_K'(\tau)
\left(\int_0^1 \frac12 H \left(1- \frac{ r}{2}, \pi_1\psi_{\H'}^{r \chi_K(\tau)}
(u_1(\tau,0))\right)\, dr \right)
\, d\tau \\
&{}&+ \int_{-2K-1}^{-2K} -\chi_K'(\tau)  \left( \int_0^1 
H\left(\frac{r}{2}, \pi_2\psi_{\H'}^{r \chi_K(\tau)} (u_2(\tau,0))\right)
\, dr\right)\, d\tau
\nonumber\\
&{}&+ \int_{2K}^{2K+1} - \chi_K'(\tau) \left( \int_0^1 H
\left(\frac{r}{2},\pi_2\psi_{\H'}^{r \chi_K(\tau)}
(u_2(\tau,0))\right)
\, dr\right)\, d\tau.
\eeastar
Recalling $\chi_K' \geq 0$ on $[-2K-1, -2K]$ and the expression \eqref{eq:psitildeH},
 the first two lines can be rewritten as
\beastar
&{}&
\int_{-2K-1}^{-2K} -\chi_K'(\tau)
\left(\int_0^{1/2}  H\left(r, \pi_1\psi_{\chi_K(\tau)\H'}^{r}(\psi_{\H'}^{\chi_K(\tau)})^{-1}
(u_1(\tau,0))\right)
\, dr \right)\, d\tau\\
&{}& +
\int_{2K}^{2K+1} -\chi_K'(\tau)
\left(\int_0^{1/2}  H\left(r, \pi_1 \psi_{\chi_K(\tau)\H'}^{r}(\psi_{\H'}^{\chi_K(\tau)})^{-1}
(u_1(\tau,0)\right)
\, dr \right)\, d\tau\\
& \leq &
\int_0^{1/2} ( \max H_r - \min H_r )\, dr
\eeastar
where we use the periodicity of $H$ assumed in Hypothesis \ref{hypo:bdy-flat}.
Similarly since $\chi_K' \leq 0$ on $[2K,2K+1]$, the second 2 lines can be bounded by

\beastar
&{}&  \int_{-2K-1}^{-2K} -\chi_K'(\tau)  \left( \int_0^1 H\left(r,\pi_2\psi_{\H'}^{\chi_K(\tau) r}
(u_2(\tau,0))\right)
\, dr\right)\, d\tau
\nonumber\\
&{}& \quad + \int_{2K}^{2K+1} - \chi_K'(\tau) \left( \int_0^1 H\left(r,\pi_2\psi_{\H'}^{\chi_K(\tau) r}
(u_2(\tau,0))\right)
\, dr\right)\, d\tau \\
& \leq &
\int_{0}^1 ( \max H_r - \min H_r )\, dr.
\eeastar
By adding the two, we have
$$
E_{(\widetilde J_K,\H_K)}^\pi(u) \leq \int_{0}^{1} ( \max H_r - \min H_r )\, dr=
\|H\|.
$$
This finishes the proof.
\end{proof}

\begin{rem}\label{rem:amalgamation}
Here we could have plugged the formula \eqref{eq:liftedflow2}
into $\pi_i\psi_{\H'}^{\chi_K(\tau) r}(u_i(\tau,0))$, $i = 1,\,2$. Since it is not necessary
for  the measurement of  the absolute value of 
$$
\left|H\left(r,\pi_i\psi_{\H'}^{\chi_K(\tau) r} (u_i(\tau,0))\right)\right|,
$$
we intentionally did not do it to visualize the nature of our calculations
entering in the proof of Shelukhin's conjecture. This kind of practice would not be possible 
or at least not be easy to be translated into the framework
of pseudoholomorphic curves in symplectization. It demonstrates that our
practice lies in the purely contact  realm of contact Hamiltonian dynamics and
contact instantons in their amalgamation, which may be compared by the amalgamation of
the symplectic Hamiltonian dynamics and Floer's equation  in symplectic geometry.
\end{rem}

We now complete the study of energy bounds by establishing the bound for the $\lambda$-energy as well.
The proof is a special case of \cite[Proposition 14.1]{oh:entanglement1} 
\emph{once the $C^0$-bound is established which is already done}, and so omit its proof.

\begin{prop}\label{prop:lambdaenergy-bound2}
Let $u$ be any finite energy solution of \eqref{eq:K-intro}. Then we have
\be\label{eq:lambdaenergy-bound-J'}
E^\perp (u_K) \leq E_-(H) + E_+(H)
\ee
where we set
$$
E_+(H) := \int_0^1 \max\{0, \max H_t\}\, dt, \quad E_-(H) := \int_0^1 \max\{0,- \min H_t\} \,dt.
$$
\end{prop}

Combining the two, we have established the fundamental uniform a priori energy
bound for the total energy
$$
E(u_K) = E^\pi( u_K) + E^\perp( u_K) \leq C < \infty
$$
where we can choose $C = \|H\| + E_-(H) + E_+(H)  +1 < \infty$.

An immediate corollary of this energy estimate, the $C^0$-estimate and argument used in
\cite{oh:entanglement1} is the following analog to Theorem \ref{thm:shelukhin-intro} applied to
the compact Legendrian submanifold $\Gamma_{\text{\rm id}}$ of the \emph{noncompact} contact manifold
$M_Q = Q \times Q \times \R$ which is tame \emph{only on $\{\eta > 0 \}$}.

\begin{thm}[Shelukhin's conjecture]\label{thm:1/2-translated}
Assume $(M,\xi)$ is a compact contact manifold and let $\psi \in \Cont_0(M,\xi)$.
Let $\lambda$ be a contact form of $\xi$ and denote by $T(M,\lambda) > 0$ the minimal
period of closed Reeb orbits of $\lambda$. Then for any $H \mapsto \psi$, the following holds:
\begin{enumerate}
\item Provided $\|H\| \leq T(M,\lambda)$, we have
$$
\# \Fix_\lambda^{\text{\rm trn}}(\psi) \neq \emptyset
$$
\item Provided $\|H\| < T(M,\lambda)$ and $H$ is nondegenerate, we have
$$
\# \Fix_\lambda^{\text{\rm trn}}(\psi) \geq \dim H^*(M; \Z_2).
$$
\end{enumerate}
\end{thm}

Once the energy estimates and the $C^0$-estimates are established, we can apply the whole
package  of deformation-cobordism analysis developed in \cite{oh:entanglement1} to conclude the following
by the same way we did for the proof of Theorem \ref{thm:shelukhin-intro}. For the readers' convenience and
for the self-containedness of the paper, we summarize the cobordism analysis
in Appendix \ref{sec:recollection}.

\part{Appendix}

\appendix

\section{The ideal boundary of a Liouville manifold}
\label{sec:liouville}

In this appendix, we briefly recall the basic definitions and properties around 
the Liouville manifolds and their ideal boundaries that are used in our discussion
on the functorial definition of contact product.  Our main reference is \cite{cieliebak-eliashberg},
\cite{giroux}.

A Liouville manifold is an exact tame symplectic manifold endowed with a Liouville form $\alpha$
satisfying two axioms:
\begin{enumerate}
\item $\omega = d\alpha$ is a symplectic form,
\item The Liouville vector field $X$ defined by $X \rfloor d\alpha = \alpha$ is complete.
\end{enumerate}

\begin{defn}[Liouville rays and the ideal boundary] Let $(M,\alpha)$ be Liouville manifold.
We call a proper Liouville trajectory defined on $[0,\infty)$ 
(resp. on $(-\infty,0]$) a \emph{positive  (resp. negative) Liouville ray}.  
We say two Liouville rays $\ell_1$ and $\ell_2$ are equivalent if  they overlap at $\infty$, i.e.,
if there exists a  time shift $t_0 > 0$ such that
$$
\ell_2(t) = \ell_1(t - t_0) \quad \text{\rm or} \quad \ell_1(t) = \ell_2(t - t_0)
$$
for all $t \geq t_0$.  We call the set of equivalence classes of Liouville rays 
the \emph{ideal boundary} and denote it by $\del_\infty M$.
\end{defn}
We have decomposition
$$
\del_\infty M = \del_{+\infty}M \sqcup \del_{-\infty}M
$$
where one of $\del_{+\infty}M$ or  $\del_{-\infty}M$ could be empty. When one of 
them is empty, we call $M$ as a positive 
(resp. negative) \emph{Liouville filling} of the contact manifold $\del_{+\infty}M$
if $\del_{-\infty}M = \emptyset$ (resp. $\del_{-\infty}M$ if $\del_{+\infty}M = \emptyset$.

The ideal boundary carries a canonical contact structure (but not a contact form!),
which we denote by $\xi_\infty$. Then $(\del_\infty M, \xi_\infty)$ becomes
a cooriented contact manifold.

\begin{defn}
A positive (resp. negative) \emph{cross section} of the Liouville flows is a contact-type hypersurface 
$S \subset M$ that satisfies the following:
\begin{enumerate}
\item  It admits a Liouville  embedding  $S \times [0,\infty) \hookrightarrow M$ of 
the half of the symplectization $S \times [0,\infty)$.
\item The Liouville flow induces a contact diffeomorphism between $(S, \xi_S)$ and 
$(\del_\infty M, \xi_\infty)$.
\end{enumerate}
\end{defn}

\section{Derivation of contact Hamiltonian vector fields}
\label{sec:XH-formula}

Recall we can decompose $X_\H$ into
\be\label{eq:XHH2}
X_\H = \left(X_1^\pi, X_2^\pi, c \frac{\del}{\del \eta}\right) + a S_\CA + b R_\CA
\ee
for any Hamiltonian $\H = \H(t,x,y,\eta)$,
and would like to determine $a, \, b, \, c$ and $X_i^\pi$ for $i = 1, \, 2$ separately.
In general, we have
$$
d\H = d_x\H + d_y \H + \frac{\del\H}{\del \eta} \, d\eta;
$$
If we further assume $\H = \pi^*F$ with $F_t = F_t(x,y)$ independent of $\eta$,
then $\frac{\del\H}{\del \eta} = 0$ and hence 
\be\label{dHH}
d\H = d_x\H + d_y \H.
\ee
In this appendix, we give another proof of Proposition \ref{prop:XHH}
which is based on the contact Hamiltonian calculus on $(M_Q,\CA)$.

\begin{prop} Let $\CA = - e^{{\frac{\eta}{2}}} \pi_1^*\lambda + e^{-{\frac{\eta}{2}}} \pi_2^*\lambda$ 
and $\H = \pi^*F$ be as above. Then Proposition \ref{prop:XHH} holds.
\end{prop}
\begin{proof}
Recall the expression \eqref{eq:dA} of $d\CA$
and the equations
$$
S_\CA \rfloor \CB = 1, \quad R_\CA \rfloor \CB = 0
$$
with $\CB: = e^{{\frac{\eta}{2}}} \pi_1^*\lambda + e^{-{\frac{\eta}{2}}} \pi_2^*\lambda$.
Substituting \eqref{eq:dA} and \eqref{eq:XHH2} into $X_\H^\pi \rfloor d\CA$ and $R_\CA[\H] \CA$
respectively, we obtain
\beastar
X_\H^\pi \rfloor d\CA & = & -e^{\frac{\eta}2} X_1^\pi \rfloor \pi_1^*d\lambda 
+ e^{-{\frac{\eta}{2}}} X_2^\pi \rfloor \pi_2^*d\lambda
+ \frac{a}{2}  d\eta - \frac{c}2 \CB \\
& = & -e^{\frac{\eta}2} X_1^\pi \rfloor \pi_1^*d\lambda 
+ e^{-{\frac{\eta}{2}}} X_2^\pi \rfloor \pi_2^*d\lambda
+ \frac{a}{2}  d\eta -\frac{c}{2} (e^{-\frac{\eta}{2}} \pi_1^*\lambda 
+ e^{-\frac{\eta}{2}} \pi_2^*\lambda),
\eeastar
and
$$
R_\CA[\H] \CA = R_\CA[\H] (-e^{-\frac{\eta}{2}}\pi_1^*\lambda + e^{\frac{\eta}{2}} \pi_2^*\lambda).
$$
By adding up the two and comparing it with 
 $d \H = X_\H^\pi \rfloor d\CA + R_\CA[\H] \CA$, we get
\bea\label{eq:dHH}
d\H  
& = &  -e^{\frac{\eta}2} X_1^\pi \rfloor \pi_1^*d\lambda  
 + \left(\frac{c}{2} + R_\CA[\H]\right)(e^{-\frac{\eta}{2}}\pi_1^*\lambda)
\nonumber\\
&{}& + e^{-\frac{\eta}2} X_2^\pi \rfloor \pi_2^*d\lambda
+  \left(- \frac{c}{2} + R_\CA[\H]\right)e^{\frac{\eta}{2}}\pi_2^*\lambda + 
 \frac{a}{2}  \, d\eta.
\eea
Then  comparing this formula \eqref{eq:dHH} with 
$d \H = X_\H^\pi \rfloor d\CA + R_\CA[\H] \CA$  we obtain
$$
\frac{a}{2} =  \frac{\del \H}{\del \eta} = 0.
$$
We augment this with the formula $X_\H \rfloor \CA = - \H$ which gives rise to
$$
b = - \H.
$$
We set 
$$
M_Q = Q \times Q \times \R
$$
and consider the natural projections 
\beastar 
\text{\rm pr}_1 \times \text{\rm pr}_2 : SQ \times SQ \to Q \times Q, & \quad & 
p: SQ \times SQ \to Q \star Q, \\
 \pi_1 \times \pi_2: M_Q \to Q \times Q,  & \quad & \text{\rm pr}: = [\text{\rm pr}_1 \times \text{\rm pr}_2]: Q \star Q \to Q \times Q
 \eeastar
 where $\text{\rm pr}: SQ \times SQ \to Q \times Q$ is induced by the natural projection 
 $$
 \text{\rm pr}_1 \times \text{\rm pr}_2 : SQ \times SQ \to Q \times Q.
 $$
 
We visualize the above projections by
 the associated commutative diagram
\be\label{eq:lifting-diagram1}
\xymatrix{
 {}&  SQ \times SQ  \ar[dl]_{p_{\frac12}}\ar@{->}'[d][dd]_(.2){\text{\rm pr}_1 \times \text{\rm pr}_2}  \ar[dr]^p&{}\\
 M_Q \ar[dr]_{\pi_1 \times \pi_2}\ar[rr]^(.7){i_{\frac12}} \ & &Q \star Q \ar[dl]^{\text{\rm pr}} \\
 &Q \times Q &
}
\ee
starting from the canonical projection with $\pi = \pi_{SQ}: SQ = Q \times \R \to (Q,\lambda)$, where
$SQ$ is the cylindrical symplectization
$SQ : = Q \times \R$ equipped with coordinates $(x,s)$  and
the symplectic form
\be\label{eq:omega+}
\omega_+ = d(e^s \pi_{SQ}^*\lambda).
\ee
Now we further specialize to the case  of the function $\H = (\pi_1 \times \pi_2)^*F$
with 
$$
F = -\widetilde H(t,x) + H(t,y)
$$
defined on $Q \times Q$.
Then on $(M_Q = Q \times Q \times \R, \CA)$, we  compute
\beastar
d\H & = & d_x\H + d_y \H = - \pi_1^*(d\widetilde H) + \pi_2^*(dH) \\
& = & \pi_1^*(X_{- \widetilde H}^\pi \rfloor d\lambda + R_\lambda[- \widetilde H] \lambda)
+ \pi_2^*(X_H^\pi \rfloor d\lambda + R_\lambda[H] \lambda) \\
& = &  \left(\pi_1^*(X_{- \widetilde H}^\pi \rfloor d\lambda)  + \pi_2^*(X_H^\pi \rfloor d\lambda) \right)
+ \left(R_\lambda[- \widetilde H] \circ \pi_1 \, \pi_1^*\lambda 
 + R_\lambda[H] \circ \pi_2 \, \pi_2^* \lambda\right).
\eeastar
We recall the decomposition 
$$
T_{(x,y,\eta)} M_Q = \widetilde \xi_1 \oplus \widetilde \xi_2 \oplus \R \langle S_\CA \rangle \oplus 
\left\langle \frac{\del}{\del \eta} \right\rangle \oplus \R \langle R_\CA \rangle.
$$
We represent a vector field on $M_Q = Q \times Q \times \R$ by the triple
$(X, Y, a \frac{\del}{\del \eta})$ where $X, \, Y$ vector fields on $Q$.
To avoid confusion, we denote by $\widetilde X$, $\widetilde Y$ the projectible vector fields on
$M_Q$ given by
$$
\widetilde X = (X, 0, 0), \quad \widetilde Y = (0,Y,0)
$$
in the following calculations appearing in the proof of the following lemma.

\begin{lem} 
 We have
\beastar
&{}& 
 \left(\pi_1^*(X_{- \widetilde H}^\pi \rfloor d\lambda)
 + \pi_2^*(X_H^\pi \rfloor d\lambda)\right) (t, x,y,\eta) \\
& = & \widetilde{X_{- \widetilde H}^\pi} \rfloor \pi_1^*d\lambda
 + \widetilde{X_H^\pi} \rfloor \pi_1^*d\lambda.
 \eeastar
\end{lem}
\begin{proof} We first note
$$
\left(\pi_1^*(X_{- \widetilde H}^\pi \rfloor d\lambda)
 + \pi_2^*(X_H^\pi \rfloor d\lambda)\right) = 
( \pi_1 \times \pi_2)^*\left((X_{- \widetilde H}^\pi \rfloor d\lambda) \oplus
(X_H^\pi \rfloor d\lambda)\right)
$$
where  the form $(X_{- \widetilde H}^\pi \rfloor d\lambda) \oplus
(X_H^\pi \rfloor d\lambda)$ is a one-form defined on $Q \times Q$.
We evaluate
\beastar
&{}&
( \pi_1 \times \pi_2)^*\left((X_{- \widetilde H}^\pi \rfloor d\lambda) \oplus
(X_H^\pi \rfloor d\lambda)\right)(X,Y,a\frac{\del}{\del \eta}) \\
& = & (X_{- \widetilde H}^\pi \rfloor d\lambda) \oplus
(X_H^\pi \rfloor d\lambda)(d(\pi_1 \times \pi_2)(X,Y,a\frac{\del}{\del \eta}) \\
& = & (\widetilde{X_{- \widetilde H}^\pi} \rfloor \pi_1^*d\lambda) +
(\widetilde{X_H^\pi} \rfloor \pi_2^* d\lambda)(X,Y,a\frac{\del}{\del \eta})
\eeastar
for all $X, \, Y$ and $a$. This shows
\beastar
&{}&
( \pi_1 \times \pi_2)^*\left((X_{- \widetilde H}^\pi \rfloor d\lambda) \oplus
(X_H^\pi \rfloor d\lambda)\right) \\
& = & \widetilde{X_{- \widetilde H}^\pi} \rfloor \pi_1^*d\lambda + 
\widetilde{X_H^\pi} \rfloor \pi_2^* d\lambda
\eeastar
which finishes the proof.
\end{proof}

By comparing this with \eqref{eq:dHH}, we obtain
$$
e^{- \frac{\eta}2} X_1^\pi = X_{\widetilde H} (t,x), \quad -e^{\frac{\eta}2} X_2^\pi = X_H(t,y) 
$$
and
\beastar
 \left(\frac{c}{2} + R_\CA[\H]\right)(-e^{-\frac{\eta}{2}} ) & =  &- R_\lambda^x[- \widetilde H]
 = - R_\lambda^x[\H] \\
 \left(- \frac{c}{2} + R_\CA[\H]\right)e^{\frac{\eta}{2}}  & = & R_\lambda^y[H]  = R_\lambda^y[\H].
  \eeastar
  From the first set of equations right above, we obtain
 \be\label{eq:Xpi12}
  X_1^\pi(t,x,y,\eta) = -e^{-\frac{\eta}2} X_{-\widetilde H}(t,x), \quad  X_1^\pi 
  = e^{\frac{\eta}2} X_{-\widetilde H}(t,y).
 \ee
  
 The second set of equations can be rewritten as
 \beastar
 \left(\frac{c}{2} + R_\CA[\H]\right)  & =  &  e^{\frac{\eta}{2}} R_\lambda^x[\H] \\
 \left(- \frac{c}{2} + R_\CA[\H]\right)  & = & e^{-\frac{\eta}{2}} R_\lambda^y[\H].
 \eeastar 

By subtracting the second from the first equation of
the second set of equations, we derive
$$
c = e^{\frac{\eta}{2}}  R_\lambda^x[\H] -  e^{-\frac{\eta}{2}}  R_\lambda^y[\H] 
= \left( e^{\frac{\eta}{2}}  R_\lambda(x) -  e^{-\frac{\eta}{2}} R_\lambda(y) \right)[\H] 
=  - 2 R_\CA[\H].
$$
Combing them all, we have finished the proof.
\end{proof}
 
\section{Contact product as an exact symplectic fibration}
\label{sec:symplectization}

 Let $(Q,\xi)$ be a contact manifold. Consider the product
$$
(S_+Q \times Q, - r \pi^*\lambda + \pi_2^*\lambda)
$$
which becomes a contact manifold. We observe that we can write
\beastar
- r \pi^*\lambda + \pi_2^*\lambda & = & \sqrt{r}\left(- \sqrt{r} \pi_1^*\lambda 
+ \frac{1}{\sqrt{r}} \pi_2^*\lambda \right) \\
& = & e^{\frac{\eta}{2}} \left(- e^{\frac{\eta}{2}} \pi_1^*\lambda 
+ e^{-\frac{\eta}{2}} \pi_2^*\lambda \right) = e^{\frac{\eta}{2}} \CA.
\eeastar

We then consider
the following lifting diagram
\be\label{eq:lifting-diagram}
\xymatrix{
 {}&  Q \star Q  \ar[dr]^\Psi &{}\\
 Q \times Q \times \R \ar[r]^{\pi_{13}}\ar[dr]_{\pi_1}\ar[ur] ^{i_s}
  & SQ \ar[d]_{\pi}
 & S_+Q \times Q \ar[dl]^{\text{\rm pr}_1} \ar[l]_{\log \circ \pi_1}\\
 &Q &
}
\ee
with $\pi = \pi_{SQ}$ starting from $(Q,\lambda)$.
Here $S_+Q$ is the conical symplectization
$S_+Q : = Q \times \R_+$ equipped with coordinates $(x,r)$  and
the symplectic form
\be\label{eq:omega+}
\omega_+ = d(r \pi_{S_+Q}^*\lambda).
\ee
Writing $r = e^\eta$, we obtain cylindrical symplectization $SQ: = Q \times \R$ e
quipped with coordinates $(x,\eta)$
and the symplectic form
\be\label{eq:omega}
\omega = d(e^\eta \pi_{SQ}^*\lambda).
\ee
In addition, the map $\log\circ \pi_1: S_+Q \times Q \to SQ$ is given by 
$$
\log\circ \pi_1((x,r),y) = (x,\log r) = (x,\eta)
$$
and the remaining maps are given as follows:
\begin{itemize}
\item The map
$\text{\rm pr}_1$ is the natural projection to the factor $S_+Q$ of $S_+Q \times Q$
followed by the canonical projection $S_+Q \to Q$.
\item  The map $\Psi: S_+Q \times Q \to Q \star Q$ is the one induced by the map
$$
id \times \pi: T^*Q \setminus \{0\} \times T^*Q \setminus \{0\} \to (T^*Q\setminus \{0\}) \times Q
$$
which naturally descends to the quotient
$$
(T^*Q \setminus \{0\} \times T^*Q \setminus \{0\})/ \sim \, \to
(T^*Q\setminus \{0\})  \times Q.
 $$
\item The map $\pi_{13}: Q \times Q \times \R \to Q \times \R = SQ$ the projection induced by the first and the third factors.
 \end{itemize}
 It is easy to check that $\Psi$ is a contactomorphism if we equip $S_+Q \times Q$ with the
 distribution given by
 $$
 \Xi_{S_+Q \times Q}: = \xi_{S_+Q} \oplus \xi.
 $$
 Furthermore they satisfy the equality
\be\label{eq:prPsiis}
\text{\rm pr}_1\circ \Psi \circ i_s = \pi_1.
\ee
In particular, we have the formula
 \be\label{eq:Psiis}
 \Psi \circ i_{\frac12} (x,y,\eta) = \left((x, e^\eta), y\right).
 \ee

\section{Recollection of the main analytic framework of \cite{oh:entanglement1}}
\label{sec:recollection}

In this section, to provide more mathematical perspective for the future purpose,
we first provide the analytic framework for general pair $(R_0,R_1)$, and then specialize to the
case $R_0 = R = R_1$ which is the case of our main interest.

\subsection{Setting up a pointed parameterized moduli space}
\label{subsec:cobordism}

In this section, we assume
$$
R_0 = R= R_1.
$$
In \cite[Section 8.1]{oh:entanglement1}, we take the family $H = H(s,t,x)$ given by
\be\label{eq:sH}
H^s(t,x) = \Dev_\lambda(t \mapsto \psi_H^{st})
\ee
and consider its contact Hamiltonian isotopies
$
\Psi_{s,t}: = \psi_{H^s}^t.
$
We then consider the following family of cut-off functions to elongate the about family over
$[0,1]^2$ to $\R \times [0,1]$.
For each $K \in \R_+ =
[0,\infty)$, we define a family of cut-off functions $\chi_K:\R \to
[0,1]$ so that for $K \geq 1$, they satisfy
\be\label{eq:chiK}
\chi_K = \begin{cases} 0 & \quad \mbox{for } |\tau| \geq K+1 \\
1 & \quad \mbox{for }|\tau| \leq K.
\end{cases}
\ee
We also require
\bea
\chi_K' & \geq & 0 \quad \mbox{on }\, [-K-1,-K] \nonumber\\
\chi_K' & \leq & 0 \quad \mbox{on }\,  [K,K+1].
\eea
For $0 \leq  K \leq  1$, define $\chi_K = K \cdot \chi_1$. Note
that $\chi_0 \equiv 0$.

We then consider
associated elongated $(\tau,t)$-parameter family
$$
H_K(\tau,t,x) = \Dev_\lambda \left(t \mapsto \psi_H^{\chi_K(\tau)t}(x)\right)
$$
and write the $\tau$-developing Hamiltonian
\be\label{eq:GK}
G_K(\tau,t,x) = \Dev_\lambda\left(\tau \mapsto \Psi_{\tau,t}^K\right)
\ee
where $\Psi_{\tau,t}^K = \Psi_{\chi_K(\tau),t}$.

\begin{hypo}\label{hypo:bdy-flat}
Without loss of generality, we may assume
\be\label{eq:bdy-flat}
H \equiv 0, \quad \text{near } \, t = 0, \, 1
\ee
on $[0,1]$ by flattening the contact isotopy, and extend $H_t$ periodically
in time  to whole $\R$. For some technical reason spelled out in Remark \ref{rem:flatat1/2},
we will also assume the flatness near $t = 1/2$ too.
\end{hypo}

Then we consider the 2-parameter perturbed contact instanton equation \eqref{eq:K} given by
\be\label{eq:K-appendix}
\begin{cases}
(du - X_{H_K}(u)\, dt - X_{G_K}(u)\, ds)^{\pi(0,1)} = 0, \\
d\left(e^{g_K(u)}(u^*\lambda + u^*H_K dt + u^*G_K\, d\tau) \circ j\right) = 0,\\
u(\tau,0) \in R,\, u(\tau,1) \in R.
\end{cases}
\ee
where $g_K(u)$ is the function on $\Theta_{K_0+1}$ defined by
\be\label{eq:gKu}
g_K(u)(\tau,t): =  g_{(\psi_{H_K}^t)^{-1}}(u(\tau,t))
\ee
for $0 \leq K \leq K_0$. We note that
if $|\tau| \geq K +1$, the equation becomes
\be\label{eq:contacton}
\delbar^\pi u = 0, \, \quad d(u^*\lambda \circ j) = 0.
\ee
We define the parameterized moduli space
\be\label{eq:M-para-K0}
\CM_{[0,K_0]}^{\text{\rm para}}(M,R;J,H) = \bigcup_{K \in [0, K_0]} \{K\} \times
\CM_K(M,R;J,H)
\ee
whose elements are defined on the domain
\be\label{eq:Z-K}
\Theta_{K_0+1}: = \Theta_- \#_{K_0+1} (\R \times [0,1]) \#_{K_0+1} \Theta_+ \subset \C
\ee
and equip it with the natural complex structure induced from $\C$.
(See Equation (8.5) and (8.6) \cite{oh:entanglement1}.)
We can also decompose $\Theta_{K_0+1}$ into the union
\be\label{eq:Theta-K0+1}
\Theta_{K_0+1} := D^- \cup  [-2K_0 -1, 2K_0+1] \cup  D^+
\ee
where we denote
\be\label{eq:D+-}
D^\pm  = D^\pm_{K_0}:= \{ z\in \C  \mid \vert z\vert \le 1, \, \pm \text{\rm Im}(z) \leq 0 \} \pm (2K_0+1)
\ee
respectively.

The following a priori $\pi$-energy identity is a key ingredient in relation to
the lower bound of the Reeb-untangling energy. We recall
the definition of oscillation
$$
\osc(H_t) = \max H_t - \min H_t.
$$
\begin{prop}[Theorem 1.26 \cite{oh:entanglement1}]
\label{prop:energy-estimate-H}
Let $u$ be any finite energy solution of \eqref{eq:K-appendix}. Then we have
\be\label{eq:pi-energy-bound-K}
E_{(J_K,H)}^\pi(u) \leq \int_0^1 \osc(H_t) \, dt = :\|H\|
\ee
and
\be\label{eq:lambdaenergy-bound-K}
E^\perp_{(J_K,H)}(u) \leq E_-(H) + E_+(H)
\ee
for all $K \geq 0$.
\end{prop}

Let
$$
E_{J_{K_\alpha},H}(u) = E_{J_{K_\alpha},H}^\pi(u) +
E_{J_{K_\alpha},H}^\perp(u)
$$
be the total energy.
With the $\pi$-energy bound \eqref{eq:pi-energy-bound-K} and the $\lambda$-energy bound \eqref{eq:lambdaenergy-bound-K},
we are now ready to make a deformation-cobordism analysis of
$$
\CM^{\text{\rm para}}_{[0,K_0+1]}(M,\lambda;R,H).
$$
We consider the $n+1$ dimensional component of the parameterized moduli space
$$
\CM_{[0,K_0]}^{\text{\rm para}}(M,R;J,H) = \bigcup_{K \in [0, K_0]} \{K\} \times
\CM_K(M,R;J,H)
$$
continued from $\CM_0(M,R;J,H) \cong R$.

Now we consider a pointed parameterized moduli space
$$
\CM_{(0,1);[0,K_0]}^{\text{\rm para}}(M,R;J,H) := \bigcup_{K \in [0,K_0]}
\CM_{1;K}(M,R;J,H).
$$
More explicitly, we have
$$
\CM_{(0,1);K}(M,R;J,H) = \{(u,z) \mid u \in , \CM_{K}(M,R;J,H), z \in \del \Theta_K\}.
$$
This is an $n+1$ dimensional smooth manifold. (See  \cite{oh:contacton-transversality} 
for the relevant transversality results.)
We also have the natural evaluation map
\be\label{eq:bdy-Ev}
\Ev^\del : \CM_{(0,1);K}(M,R;J,H) \to R \times \R_+\times \del \Theta_K ;
\quad (u, K,z) \mapsto (u(z), K, z)
\ee
and the relevant evaluation transversality result is also proved in \cite{oh:contacton-transversality}.

We  consider the gauge transformation of $u$
\be\label{eq:ubarK}
\overline u_K(\tau,t): = \left((\psi_H^{\chi_K(\tau)t})^{-1}(u(\tau,t)\right)
\ee
and make the choice of the family
\begin{choice}[Choice 9.6 \cite{oh:entanglement1}]\label{choice:lambda-J}
 We consider the following two parameter families of $J'$ and $\lambda$:
\bea
J_{(s,t)}' & =  & (\psi_{H^s}^t)_*J \label{eq:J-prime-appendix}\\
\lambda_{(s,t)}' & = & (\psi_{H^s}^t)_*\lambda. \label{eq:lambda-prime-appendix}.
\eea
\end{choice}

A straightforward standard calculation also gives rise to the following.

\begin{lem}[Lemma 6.7 \cite{oh:entanglement1}]\label{lem:Ham-Reeb} For given $J_t$, consider $J'$ defined as
above. We equip $(\Sigma,j)$ a K\"ahler metric $h$. Let $g_H(u)$ be the
function defined in \eqref{eq:gKu}.
Suppose $u$ satisfies \eqref{eq:perturbed-contacton-bdy}
with respect to $J_t$. Then $\overline u$ satisfies
\be\label{eq:contacton-bdy-J0-appendix}
\begin{cases}
\delbar^\pi_{J'} \overline u = 0, \quad d(\overline u^*\lambda \circ j) = 0 \\
\overline u(\tau,0) \in \psi_H^1(R), \, \overline u(\tau,1) \in R
\end{cases}
\ee
for $J$. The converse also holds. And $J' = J'(\tau,t)$ satisfies
$
J'(\tau,t) \equiv J_0
$
for $|\tau|$ sufficiently large.
\end{lem}

Then $\overline u_K$ also satisfies the energy identity
$$
E_{(J_K,H)}^\pi(u) = E_{J'}^\pi(\overline u), \quad   E_{(J_K,H)}^\perp(u): = E^\perp_{J'}(\overline u).
$$
(See \cite[Proposition 7.7]{oh:entanglement1}.)

\begin{prop}\label{prop:energy-estimate}
Let $u$ be any finite energy solution of \eqref{eq:K-appendix}. Then we have
\be\label{eq:lambdaenergy-bound}
E^\pi_{J'}(\overline u_K) \leq \|H\|, \quad E^\perp_{J'}(\overline u_K) \leq E_-(H) + E_+(H).
\ee
\end{prop}

\begin{rem}\label{rem:fixed-intersect} We will denote by $\CM^{\text{\rm para}}_{[0,K_0+1]}(M,\lambda;R,H)$
both the set of solutions of \eqref{eq:K-appendix} and that of its gauge transform \eqref{eq:contacton-bdy-J0-appendix}
freely without distinction throughout the paper. However for the clarity, we will always denote by $U$
a solution of the former equation and by $\overline U$ that of the latter later when
we work with the product $M_Q = Q \times Q \times \R$. We would like to mention that
the $\Z_2$ anti-symmetry are easier to see and more apparent for \eqref{eq:K-appendix}. Because the gauge
transform destroys the $\Z_2$ symmetry, we will work with \eqref{eq:K-appendix} to exploit the symmetry
while we will work with  \eqref{eq:contacton-bdy-J0-appendix} for  the energy estimates which is
easier to work with for the latter purpose.
\end{rem}

\subsection{Gromov-type compactness theorem of contact instanton moduli space}

We will focus on the general case without nondegeneracy hypothesis because
the proof of nondegeneracy case is not much different from that of
\cite[Part 2]{oh:entanglement1}.

In this section, we consider a general pair $(M,R)$ of
compact Legendrian submanifold $R$ in a tame contact manifold
$(M,\lambda)$ in the sense of \cite{oh:entanglement1}.
Then in the next section, we will specialize to the case of
$$
R = R_0 \subset Q \times Q \times \R
$$
after we state the general compactness theorem \cite[Theorem 8.8]{oh:entanglement1}.

We first state the following which is equivalent to \cite[Corollary 9.4]{oh:entanglement1}.
\begin{lem}[Compare with Corollary 9.4 \cite{oh:entanglement1}]\label{lem:ev0}
 The evaluation map
 $$
 \ev_{0}: \CM_{(0,1);0}(M,R;J_0,H) \to R
 $$
 is a diffeomorphism.
In particular its $\Z_2$-degree is nonzero.
\end{lem}

We choose smooth embedded paths
$\Gamma: [0,1] \to R \times \R_+ \times (\R \times \{0,1\})$ with
$$
\Gamma(s) = (\gamma(s), K(s), z(s))
$$
such that
\be\label{eq:Gamma-condition}
 K(0) =0, \quad  K_0  \leq  K(1) \leq 2K_0
\ee
where $K_0 > 0$ is the constant given in  the following proposition.

\begin{prop}[Proposition 8.9 \cite{oh:entanglement1}]\label{prop:obstruction}
Let $H_t$ be the Hamiltonian such that $\psi_H^1(R) \cap Z_R$ is empty
and $H = sH_t$ and $J$ as before. Suppose $\|H\| < T_\lambda(M,R)$.
Then there exists $K_0> 0$ sufficiently large such that
$\CM_K(M,R;J,H)$ is empty for all $K \geq K_0$.
\end{prop}

Choosing a generic $\Gamma$, we can make
the map \eqref{eq:bdy-Ev} transverse to the path $\Gamma$
so that
$$
N_\Gamma:= \Ev ^{-1}(\Gamma)
$$
becomes a one dimensional manifold with its boundary consisting of
$$
\CM_{K(0)}(M,R;J,H) \times \{z(0)\} \coprod \CM_{K(1)}(M,R;J,H)
\times \{z(1)\}.
$$
We \emph{fix} a generic point $x \in \psi(R)$ and choose $\Gamma$ whose image is
as small as we want near the given point $x$.
(See \cite{oh:contacton-transversality} for the relevant mapping and the evaluation transversality results.)

For given constant $C> 0$ we consider the set of Hamiltonians with the bound
\be\label{eq:standing-hypothesis-H}
\|H\| <  C < \infty
\ee
such that $\psi(R) \cap Z_R = \emptyset$ which is equivalent to
\be\label{eq:standing-hypothesis-psi}
\mathfrak{Reeb}(\psi(R), R) = \emptyset.
\ee
Under these two conditions, there exists some $K_0 > 0$ such that
\be\label{eq:K_0-empty}
\CM_{K}(M,R;J,H) = \emptyset
\ee
for all $K \geq K_0$ by the same proof of as that of Proposition \ref{prop:obstruction}.
Obviously this implies that the same holds for the pointed moduli space
$$
\CM_{(0,1);K}(M,R;J,H) = \emptyset
$$
Choose any sequence $K_\alpha \to \infty$ as $\alpha \to \infty$ and $u_\alpha \in \CM_{K_\alpha}(M,R; J,H)$.
Select a sequence  still denoted by $K_\alpha$  with $K_\alpha \to \infty$ and elements
$$
   u_\alpha \in \CM_{K_\alpha}(M,R;J,H).
$$
By the energy estimate given in \eqref{eq:pi-energy-bound},
we have
$$
E^\pi(u_\alpha) \leq  \|H\|
$$
for all $\alpha \to \infty$. And we also have the uniform bound of the vertical energy from
\eqref{eq:lambdaenergy-bound}.

Under the assumption \eqref{eq:standing-hypothesis-psi},
the above boundary is a {\it single point}, i.e, $(u_0, z(0))$ where $u_0 \equiv \gamma(0)$ is the constant map.
Therefore $N_\Gamma$ cannot be compact.
The only source of non-compactness of $N_\Gamma$ lies in the
bubbling off of either a spherical contact instanton of $J_{(K_1,z_1)}$
for some $z_1 \in \Theta_{K_1}$, or a disc-type open contact instanton with boundary on $R$.

Then the following Gromov-Floer-Hofer type convergence result directly applied to \eqref{eq:K-appendix}
unlike in \cite{oh:entanglement1}, where it is applied to the intersection theoretic version $\overline u$ instead of
$u$ itself, can be proved in the same way: This is because perturbed contact instanton is
a (nonlinear) zero order perturbation of contact instanton and we have established the a priori elliptic estimates  in \cite{oh:perturbed-contacton} for the perturbed equation. Once this is done, the same bubbling
argument utilized in \cite[Part 3]{oh:entanglement1} applies with little modification similarly
as the same bubbling argument applied to Gromov's pseudoholomorphic curves also applies to
Floer's Hamiltonian perturbed pseudoholomorphic curves.

\begin{thm}[Compare with Theorem 1.30 \cite{oh:entanglement1}]\label{thm:bubbling}
Consider the moduli space $\CM^{\text{\rm para}}(M,R;J,H)$.
Then one of the following alternatives holds:
\begin{enumerate}
\item
There exists some $ C  > 0$ such that
\be\label{eq:dwC0-intro}
\|d u\|_{C^0;\R \times [0,1]} \leq C
\ee
where $C$ depends only on $(M,R;J,H)$ and $\lambda$.
\item There exists a sequence $u_\alpha \in \CM_{K_\alpha}(M,R;J,H)$ with $K_\alpha \to K_\infty \leq K_0$
and a finite set $\{\gamma_j^+\}$ of closed Reeb orbits of $(M,\lambda)$ such that $u_\alpha$
weakly converges to the union
$$
u_\infty =  u_{-,0} + u_0 +  u_{+,0} + \sum_{j=1} v_j + \sum_k w_k
$$
in the Gromov-Floer-Hofer sense, where
$$
u_0 \in \CM_{K_\infty}(M,R;J,H),
$$
$$
v_j \in \CM(M,J_{z_j};\alpha_j); \quad \alpha_j \in \mathfrak{Reeb}(M,\lambda),
$$
and
$$
w_k \in \CM(M,R, J_{z_j};\mathfrak{B}_k); \quad \mathfrak{B}_k \in \mathfrak{Reeb}(M,R;\lambda).
$$
\end{enumerate}
Here the domain point $z_j \in \del \dot \Theta_{K_\infty +1}$ is the point at which the corresponding bubble is attached.
\end{thm}

\bibliographystyle{amsalpha}

\bibliography{biblio2}

\end{document}